\documentclass[12pt]{article}
\usepackage{amsmath}
\usepackage{latexsym}
\usepackage{amssymb}
\newtheorem{thm}{Theorem}[section]
\newtheorem{la}[thm]{Lemma}
\newtheorem{Defn}[thm]{Definition}
\newtheorem{Remark}[thm]{Remark}
\newtheorem{prop}[thm]{Proposition}
\newtheorem{cor}[thm]{Corollary}
\newtheorem{Example}[thm]{Example}
\newtheorem{Number}[thm]{\!\!}
\newenvironment{defn}{\begin{Defn}\rm}{\end{Defn}}
\newenvironment{example}{\begin{Example}\rm}{\end{Example}}
\newenvironment{rem}{\begin{Remark}\rm}{\end{Remark}}
\newenvironment{numba}{\begin{Number}\rm}{\end{Number}}
\newenvironment{proof}{{\noindent\bf Proof.}}%
                  {\nopagebreak\hspace*{\fill}$\Box$\medskip\medskip\par}
\newcommand{\Punkt}{\nopagebreak\hspace*{\fill}$\Box$}
\newcommand{\wb}{\overline}

\newcommand{\at}{\symbol{'100}}
\newcommand{\wt}{\widetilde}
\newcommand{\impl}{\Rightarrow}
\newcommand{\mto}{\mapsto}

\newcommand{\isom}{\cong}
\DeclareMathOperator{\Ad}{Ad}
\newcommand{\N}{{\mathbb N}}
\newcommand{\K}{{\mathbb K}}
\newcommand{\Q}{{\mathbb Q}}
\newcommand{\Z}{{\mathbb Z}}
\DeclareMathOperator{\ad}{ad}
\newcommand{\ch}{{\mathfrak h}}
\newcommand{\pl}{{\displaystyle \lim_{\longleftarrow}}}
\DeclareMathOperator{\Aut}{Aut}
\newcommand{\sub}{\subseteq}
\DeclareMathOperator{\GL}{GL}
\DeclareMathOperator{\id}{id}
\DeclareMathOperator{\BS}{BS}
\newcommand{\sbull}{{\scriptscriptstyle \bullet}}
\newcommand{\aeq}{\Leftrightarrow}
\DeclareMathOperator{\dv}{div}
\DeclareMathOperator{\tor}{tor}
\begin{document}
$\;$\\[-27mm]
\begin{center}
{\Large\bf Expansive automorphisms of\vspace{1mm} totally disconnected, locally compact groups}\\[7mm]
{\bf Helge Gl\"{o}ckner and C.R.E. Raja}\vspace{4mm}
\end{center}
\begin{abstract}\vspace{.3mm}\noindent
We study automorphisms~$\alpha$ of a totally disconnected, locally
compact group~$G$ which are \emph{expansive} in the sense that
$\bigcap_{n \in \Z} \alpha^n(U)=\{1\}$ for some identity
neighbourhood $U\sub G$. Notably, we prove that the automorphism
induced by~$\alpha$ on a quotient group $G/N$ of $G$ modulo an
$\alpha$-stable closed normal subgroup $N$ is always expansive.
Further results involve the contraction groups $U_\alpha:=\{g\in
G\colon \mbox{$\alpha^n(g)\to 1$ as $n\to\infty$}\}$. If $\alpha$ is
expansive, then $U_\alpha U_{\alpha^{-1}}$ is an open identity
neighbourhood in~$G$. We give examples where $U_\alpha
U_{\alpha^{-1}}$ fails to be a subgroup. However, $U_\alpha
U_{\alpha^{-1}}$ is an $\alpha$-stable, nilpotent open subgroup
of~$G$ if~$G$ is a closed subgroup of~$\GL_n(\Q_p)$.
Further results are devoted to the divisible and torsion parts of~$U_\alpha$,
and to the so-called ``nub''
$U_0=\wb{U_\alpha}\cap \wb{U_{\alpha^{-1}}}$
of an expansive automorphism.\vspace{3mm}
\end{abstract}
%1
%
{\footnotesize {\em Classification}:
22D05 (primary); %gen prop struc lcp gps
22D45, % auto gps of lcp gps
22E20, % General properties and structure of other Lie groups
37A25, % ergodicity, mixing
37P20 (secondary)\\[2mm] % non-arch dyn systems over local fields
}
\section*{Introduction and statement of results}
We consider automorphisms $\alpha\colon G\to G$ of a totally
disconnected locally compact topological group~$G$ which are
\emph{expansive} in the sense that $\bigcap_{n \in \Z}
\alpha^n(U)=\{1\}$ for some identity neighbourhood $U\sub G$.
Expansive automorphisms of totally disconnected, \emph{compact}
groups were studied by \cite{Ki}, \cite{Sc} and\linebreak recently
in \cite{Wi3}.
The importance of expansive automorphisms
for the theory of general automorphisms is highlighted by the fact
that every automorphism $\alpha$ of a totally disconnected compact group~$G$
is a projective limit
\[
\alpha=\pl\, \alpha_j\vspace{-1.3mm}
\]
of expansive automorphisms $\alpha_j$ of certain Hausdorff quotient groups $G/N_j$
of $G$ such that $G=\pl\, G/N_j$\vspace{-.7mm} (see \cite{Wi3}).
Our goal is to improve the understanding in the case
of non-compact groups. Special cases of expansive automorphisms are
automorphisms $\alpha\colon G\to G$ which are \emph{contractive} in
the sense that $\alpha^n(g)\to 1$ as $n\to\infty$ for each $g\in G$
(see Remark~\ref{contrexpa}). The structure of totally disconnected,
locally compact groups admitting contractive automorphisms was
elucidated in~\cite{GaW} (building on earlier work like \cite{Sie}
and \cite{Wan}), and the results obtained there can also be used as
tools in the investigation of expansive automorphisms (as we shall
see). Other key ingredients are the structure theory of totally
disconnected, locally compact groups (\cite{Wi1}, \cite{Wi2}), in
which contractive automorphisms
play an important role (as worked out in~\cite{BaW}).\\[4mm]
Note that every automorphism~$\alpha$ of a discrete group~$G$
(e.g., $\alpha=\id_G$) is expansive
(as we may choose $U=\{1\}$ then).
Therefore discrete groups and their automorphisms
are part of the theory of expansive automorphisms.
As a consequence, groups permitting expansive automorphisms
need not have
any particular algebraic properties.\\[4mm]
Our first main result generalizes \cite[Proposition~6.1]{Wi3}
(devoted to the case of compact groups).
As usual,
a subset $H\sub G$ is called \emph{$\alpha$-stable} if $\alpha(H)=H$.\\[4mm]
{\bf Theorem A.}
\emph{Let $\alpha\colon G\to G$ be an
automorphism of a totally disconnected, locally compact group,
$N\sub G$ be an $\alpha$-stable closed normal subgroup
and $\bar{\alpha}\colon G/N\to G/N$,
$gN\mto \alpha(g)N$ be the automorphism of $G/N$ induced by~$\alpha$.
Then $\alpha$ is expansive if and only if both $\alpha|_N$ and $\bar{\alpha}$
are expansive.}\\[4mm]
Composition series play a central role
in the study of contractive automorphisms~\cite{GaW}.
In the case of expansive automorphisms,
composition series need not exist. However, there are
series which get as close as possible to such.\\[4mm]
{\bf Theorem B.}
\emph{If $\alpha\colon G\to G$ is an expansive
automorphism of a totally disconnected, locally compact group~$G$,
then there exist $\alpha$-stable closed subgroups}
\[
G=G_0\supseteq G_1\supseteq\cdots
\supseteq G_n=\{1\}
\]
\emph{of $G$ such that $G_j$ is normal in $G_{j-1}$ for $j\in \{1,\ldots,n\}$
and
every $\alpha_j$-stable
closed normal subgroup of $G_{j-1}/G_j$
is discrete or open, where
$\alpha_j$ denotes the induced automorphism $G_{j-1}/G_j\to G_{j-1}/G_j$,
$gG_j\mto\alpha(g)G_j$.
Moreover, one can achieve that
each of the quotient groups $G_{j-1}/G_j$
is discrete, abelian or topologically perfect.}\\[4mm]
In addition, one may assume that all abelian, non-discrete
factors $G_{j-1}/G_j$ are
simple contraction groups with respect to the automorphism~$\alpha_j$
or its
inverse, or isomorphic to an infinite power
$C_p^\Z$ of a cyclic group of prime order,
endowed with the right-shift (cf.\ Remark~\ref{nicerabel}
and Proposition~\ref{abeliancase}).
Topological perfectness means that the commutator group
is dense.
For second countable groups, more detailed information on the perfect factors is available (see Remark~\ref{newremark}).
The proofs of Theorem~A and Theorem~B
hinge on the fact that there is a bound
on the number of non-discrete factors in series
for $(G,\alpha)$ (Proposition~\ref{almostmax}).\\[4mm]
According to \cite{Wi2},
a compact open subgroup $V\sub G$ is called
\emph{tidy for $\alpha$}
if it has the following properties:
\begin{itemize}
\item[{\bf T1}]
$V=V_+V_-$, where $V_+:=\bigcap_{n=0}^\infty\alpha^n(V)$
and $V_-:=\bigcap_{n=0}^\infty\alpha^{-n}(V)$;
\item[{\bf T2}]
The $\alpha$-stable subgroups $V_{++}:=\bigcup_{n\in \N_0}\alpha^n(V_+)$
and $\bigcup_{n\in \N_0}\alpha^{-n}(V_-)$ $=:V_{--}$
are closed in~$G$.
\end{itemize}
Note that $V_+\sub \alpha(V_+)\sub \alpha^2(V_+)\sub\cdots$
and $V_-\sub \alpha^{-1}(V_-)\sub \alpha^{-2}(V_-)\sub\cdots$
here. Following \cite{Wi3},
the intersection~$U_0$
of all subgroups~$V$ which are tidy for~$\alpha$
is called the \emph{nub} of~$\alpha$;
it is a compact, $\alpha$-stable subgroup of~$G$.\\[4mm]
We mention that a totally
disconnected, locally compact group
admitting an expansive automorphism~$\alpha$
is always metrizable (cf.\ \cite{Lam})
and has a second countable,
$\alpha$-stable open subgroup
(Lemma~\ref{basicexpa}\,(a)).\\[4mm]
If $\alpha\colon G\to G$ is an automorphism
of a totally disconnected, locally compact group~$G$,
then
\[
U_\alpha:=\{g\in G\colon \mbox{$\alpha^n(g)\to 1$ as $n\to\infty$}\}
\]
is a subgroup of~$G$ called the associated \emph{contraction
group}. In general, $U_\alpha$ need not be closed.
However, if~$\alpha$ is expansive,
then the topology on~$U_\alpha$
``can be made locally compact,''
i.e., it can be
refined to a totally disconnected,
locally compact group topology~$\tau^*$
with respect to which~$\alpha|_{U_\alpha}$
is contractive
(see \cite[Proposition~9]{Si2} for this fact, or our Lemma~\ref{expalcp}).
In this way, the structure theory of locally compact
contraction groups (see \cite{Sie},
\cite{Wan} and \cite{GaW})
becomes available.
In particular, the set
\[
T_\alpha:=\tor(U_\alpha)
\]
of all torsion elements in the group~$U_\alpha$
and the set
\[
D_\alpha:=\dv(U_\alpha)
\]
of divisible elements
are $\alpha$-stable closed subgroups
of $(U_\alpha,\tau^*)$,
and $(U_\alpha,\tau^*)=D_\alpha\times T_\alpha$
internally as a topological group,
if we endow $D_\alpha$ and $T_\alpha$
with the topology induced by $(U_\alpha,\tau^*)$
(see \cite[Theorem~B]{GaW}).\\[4mm]
Recall that the closure
$\wb{U_\alpha}$ of the contraction group $U_\alpha$ in~$G$
plays a role in the structure
theory of totally disconnected,
locally compact groups; for example, the scale of $\alpha^{-1}$
can be calculated on $\wb{U_\alpha}$ (see \cite[Proposition~3.21\,(3)]{BaW}).
Our next theorem provides information on $\wb{U_\alpha}$,
and on the divisible part~$D_\alpha$~of~$U_\alpha$.\\[4mm]
{\bf Theorem C.}
\emph{Let $G$ be a totally disconnected, locally compact group
and $\alpha\colon G\to G$ be
an automorphism such that $U_\alpha$ can be made locally compact
$($for example, any expansive automorphism$)$. Then
$\wb{U_\alpha}=D_\alpha\times \wb{T_\alpha}$} (\emph{internally})
\emph{as a topological group,
and $\wb{T_\alpha}=T_\alpha U_0$.
In particular,
$D_\alpha=\dv(U_\alpha)$
is an $\alpha$-stable} closed \emph{subgroup of~$G$,
and both the closure $\wb{T_\alpha}$ of~$T_\alpha$
in~$G$
and the nub~$U_0$\linebreak
centralize~$D_\alpha$.}\\[4mm]
We mention that
the nub~$U_0$ of an expansive automorphism
need not have an open normalizer in~$G$
(Remark~\ref{notnormal}), in which case
not both of~$U_\alpha$ and~$U_{\alpha^{-1}}$
normalize~$U_0$.\\[4mm]
Classes of examples are also considered.
An analytic automorphism~$\alpha$ of a Lie group $G$ over a totally disconnected
local field~$\K$ is expansive if and only if the associated
Lie algebra automorphism~$L(\alpha)$ is expansive,
which means that none of its eigenvalues
in an algebraic closure has absolute value~$1$.
The Lie algebra $L(G)$ of~$G$ is then nilpotent (Proposition~\ref{algnilp}).
In the case of $p$-adic Lie groups
for a prime number~$p$,
we obtain:\\[4mm]
{\bf Theorem D.}
\emph{Let $G$ be a $p$-adic Lie group
which is \emph{linear} in the sense that there
exists an injective continuous homomorphism $G\to \GL_n(\Q_p)$
for some $n\in\N$.
Let $\alpha\colon G\to G$ be an expansive
automorphism.
Then $G$ has an open $\alpha$-stable
subgroup which is nilpotent.}\\[4mm]
If $\alpha\colon G\to G$ is an expansive automorphism of a totally
disconnected, locally compact group~$G$, then $U_\alpha
U_{\alpha^{-1}}$ is an open subset of~$G$ (see
Proposition~\ref{basicexpa}). In many examples, $U_\alpha
U_{\alpha^{-1}}$ happens to be a subgroup of~$G$ (for instance,
for all expansive automorphisms of closed subgroups
$G\sub \GL_n(\Q_p)$, see Proposition~\ref{alggp}).
However, this is not always so which
could be seen from Remark~\ref{Hei1}.  We also consider the localized
completion~$G_{p,q}$ of a Baumslag-Solitar group $\BS(p,q)=\langle
a,t\colon ta^pt^{-1}=a^q\rangle$ with primes $p\not=q$, as recently
studied in \cite{EaW}. Then $\BS(p,q)\sub G_{p,q}$.
We show:\\[4mm]
{\bf Theorem E.}
\emph{Let $\alpha\colon G_{p,q}\to G_{p,q}$ be conjugation by $t$.
Then $\alpha$ is expansive but $U_\alpha U_{\alpha^{-1}}$ is not a
subgroup of $G_{p,q}$.}\\[4mm]
\emph{Acknowledgement.} The second author was supported by the
German Academic Exchange Service (DAAD) and the second author wishes
to thank DAAD and Institut f\"{u}r Mathematik, Universit\"{a}t
Paderborn, Paderborn, Germany for providing the facilities during
his stay.
\section{Preliminaries and basic facts}\label{secprel}
We write $\N=\{1,2,\ldots\}$, $\N_0:=\N\cup\{0\}$ and
$\Z:=\N_0\cup
{-\N}$. If $J$ is a finite set, we let $\# J$ be its cardinality. We
write $X\sub Y$ for inclusion of sets, while $X\subset Y$ means
that~$X$ is a proper subset of~$Y$. As usual, we write $N\lhd G$
if~$N$ is a normal subgroup of~$G$. All topological groups
considered in this article are assumed Hausdorff, and locally
compact topological groups are simply called locally compact groups.
Totally disconnected, locally compact non-discete topological fields
(like the field of $p$-adic numbers) will be called local fields
(see \cite{Wei} for further information).  See \cite{Bou} and
\cite{Ser} for basic information on Lie groups over local (and more
general complete ultrametric) fields (which we always assume
finite-dimensional). If we say that~$\alpha$ is an automorphism of a
topological group, then we assume that both~$\alpha$ and
$\alpha^{-1}$ are continuous; similarly, both~$\alpha$
and~$\alpha^{-1}$ are assumed analytic if~$\alpha$ is an
automorphism of an analytic Lie group over a local field. We write
$\Aut(G)$ for the group of all automorphisms of a topological
group~$G$. If a subgroup~$N\sub G$ is stable under all $\alpha\in
\Aut(G)$, then~$N$ is called \emph{topologically characteristic}. A
topological group~$G$ is called \emph{topologically perfect} if its
commutator group $[G,G]$ is dense in~$G$. If $F$ is a finite group
and $X$ a set, we write $F^X:=\prod_{x\in X}F$ for the direct power
endowed with the compact product topology. By contrast, $F^{(X)}\sub
F^X$ is the subgroup of all $(g_x)_{x\in X}\in F^X$ such that
$g_x=1$ for all but finitely many $x\in X$. We shall always endow
$F^{(X)}$ with the discrete topology. Surjective, open, continuous
homomorphisms between topological
groups are called \emph{quotient morphisms}.\\[3mm]
If $\alpha\colon G\to G$
is an automorphism of a locally compact group,
choose a Haar measure $\lambda$ on~$G$ and define
the module of~$\alpha$ via
\[
\Delta_G(\alpha)
:= \lambda(\alpha(K))/\lambda(K),
\]
for any compact subset
$K\sub G$ with non-empty interior.
If $\alpha\colon G\to G$
is an automorphism of a totally disconnected, locally compact group,
we define the nub~$U_0$, the contraction group
$U_\alpha$ and its subgroups $T_\alpha$ and $D_\alpha$
as explained in the introduction.
If $V\sub G$ is a compact open subgroup,
we shall use the subgroups $V_+$, $V_{++}$,
$V_-$ and $V_{--}$ defined there,
and abbreviate
\begin{equation}\label{triv}
V_0:=V_+\cap V_-=\bigcap_{n\in \Z}\alpha^n(V).
\end{equation}
We shall also need the so-called \emph{Levi factor}
\[
M_\alpha:=\{g\in G\colon \mbox{$\{\alpha^n(g)\colon n\in \Z\}$
is relatively compact}\};
\]
it is known that $M_\alpha$
is an $\alpha$-stable closed subgroup
of~$G$ \cite[p.\,224]{BaW}.
The following lemma compiles basic facts
concerning expansive automorphisms.
\begin{la}\label{basicexpa}
If $\alpha$ is an expansive automorphism
of a totally disconnected,
locally compact group~$G$,
then the following holds:
\begin{itemize}
\item[\rm(a)]
$G$ is metrizable
and has an $\alpha$-stable,
$\sigma$-compact open subgroup;
\item[\rm(b)]
$V_{--}=U_\alpha$
and $V_{++}=U_{\alpha^{-1}}$
for each compact open subgroup $V\sub G$
such that $V_0=\{1\}$;
\item[\rm(c)]
$G$ has a compact open subgroup $V$ such that
$V=V_+V_-$ and $V_0=\{1\}$;
\item[\rm(d)]
$U_\alpha U_{\alpha^{-1}}$
is open in~$G$;
\item[\rm(e)]
$\alpha|_H$ is expansive, for each
$\alpha$-stable subgroup $H\sub G$.
\end{itemize}
\end{la}
\begin{proof}
(a) Because $\alpha$ is expansive, there exists
an identity neighbourhood~$V$ such that
$\bigcap_{n\in\Z}\alpha^n(V)=\{1\}$.
After shrinking~$V$,
we may assume that~$V$ is compact.
Because~$V$ is compact and
$(\bigcap_{k=-n}^n\alpha^k(V))_{k\in \N}$
a decreasing sequence of closed identity neighbourhoods in~$V$
with intersection~$\{1\}$,
the members of the sequence
form a basis of identity
neighbourhoods. Hence~$G$ is metrizable.
The subgroup of~$G$ generated by $\bigcup_{n\in \Z}\alpha^n(V)$
is $\alpha$-stable, open and $\sigma$-compact.

(b) By \cite[Proposition 3.16]{BaW},
we have $V_{--}=U_\alpha V_0$
and thus $V_{--}=U_\alpha$. Likewise,
$V_{++}=U_{\alpha^{-1}}$.

(c) and (d): Using expansiveness and van Dantzig's Theorem
\cite[Theorem 7.7]{HaR},
we find a compact open subgroup $W\sub G$ such that
$\bigcap_{n\in \Z}\alpha^n(W)=\{1\}$. By \cite[Lemma~1]{Wi1}, there
exists $m\in \N$ such that $V:=\bigcap_{k=1}^m\alpha^k(W)$ satisfies
$V=V_+V_-$. Then $V_0\sub W_0=\{1\}$ and thus $V_0=\{1\}$,
proving~(c). The latter entails $U_\alpha=V_{--}$ and
$U_{\alpha^{-1}}=V_{++}$, by~(b). In particular, $V_-\sub U_\alpha$
and $V_+\sub U_{\alpha^{-1}}$, entailing that $V=V_+V_-\sub U_\alpha
U_{\alpha^{-1}}$. Thus $U_\alpha U_{\alpha^{-1}}$ is an identity
neighbourhood. Given $g\in U_\alpha$ and $h\in U_{\alpha^{-1}}$, the
map $G\to G$, $x\mto gxh$ is a homeomorphism which takes $U_\alpha
U_{\alpha^{-1}}$ onto itself and $1$ to $gh$. Hence $U_\alpha
U_{\alpha^{-1}}$ has $gh$ in its interior and thus $U_\alpha
U_{\alpha^{-1}}$ is open.

(e) If $V\sub G$ is an identity neighbourhood
with $\bigcap_{n\in\Z}\alpha^n(V)=\{1\}$,
then $V\cap H$ is an identity neighbourhood in~$H$
and $\bigcap_{n\in\Z}\alpha^n(V\cap H)=\{1\}$.
\end{proof}
The first statement of
Lemma~\ref{basicexpa}\,(a)
also follows from \cite[Lemma 2.4]{Lam}.
\begin{numba}\label{recall1}
Let $\alpha$ be an automorphism of a totally disconnected,
locally compact group~$G$ and~$U_0$ be the nub of~$\alpha$.
The following facts are useful:
\begin{itemize}
\item[(a)]
The closure of $U_\alpha$ in~$G$
is
$\wb{U_\alpha}=U_\alpha U_0$
(see \cite[Corollary 3.30]{BaW}
if~$G$ is metrizable; the general
case follows with~\cite{Jaw}).
\item[(b)]
$U_0=\wb{U_\alpha}\cap\wb{U_{\alpha^{-1}}}$
(see \cite[Corollary 3.27]{BaW}
if~$G$ is metrizable; the general
case follows with~\cite{Jaw}).
\item[(c)]
$U_0\cap U_\alpha$ and $U_0\cap U_{\alpha^{-1}}$
are dense in~$U_0$
(see \cite[Theorem 4.1\,(v) and Proposition~5.4\,(i)]{Wi3}).
\item[(d)]
If $N\sub G$ is a closed
$\alpha$-stable subgroup, $q\colon G\to G/N$
the canonical quotient morphism and $\bar{\alpha}$
the automorphism of $G/N$ induced by~$\alpha$,
then $q(U_\alpha)=U_{\bar{\alpha}}$
(see \cite[Theorem 3.8]{BaW} if $G$ is metrizable,
\cite{Jaw} in the general case).
\item[(e)]
$U_\alpha$ is closed if and only if $U_{\alpha^{-1}}$ is closed,
if and only if $G$ has small subgroups tidy for~$\alpha$,
i.e., every identity neighbourhood of~$G$ contains some
tidy subgroup (see \cite[Theorem 3.32]{BaW}
if $G$ is metrizable; the general case
can be deduced using the techniques from~\cite{Jaw}).
\item[(f)]
$P_\alpha:=\{g\in G\colon \mbox{$\alpha^{\N_0}(g)$ is relatively compact}\}$
normalizes~$U_\alpha$
(see \cite[Proposition~3.4]{BaW}).
Hence $M_\alpha=P_\alpha\cap P_{\alpha^{-1}}$ and $U_0\sub M_\alpha$
normalize~$U_\alpha$.
\end{itemize}
\end{numba}
The following results help to show that certain
automorphisms are expansive.
\begin{prop}\label{givesexp}
Let $\alpha$ be an automorphism of a totally disconnected, locally
compact group $G$. Then the following holds:
\begin{itemize}
\item[\rm(a)]
$\alpha$ is expansive if and only if its restriction
$\alpha|_{M_\alpha}$ to the Levi factor is expansive.
\item[\rm(b)]
If $U_\alpha$ is closed, then $\alpha$ is expansive
if and only if $U_\alpha U_{\alpha^{-1}}$ is open in~$G$,
if and only if $M_\alpha$ is discrete.
\end{itemize}
\end{prop}
\begin{proof}
(a) In view of Lemma~\ref{basicexpa}\,(e), we only need to show that
if $\alpha|_{M_\alpha}$ is expansive, then so is~$\alpha$. Let
$P\sub M_\alpha$ be a compact, open identity neighbourhood such that
$\bigcap_{n\in \Z}\alpha^n(P)=\{1\}$. There is a compact identity
neighbourhood $Q\sub G$ such that $Q\cap M_\alpha=P$. If $g\in
I:=\bigcap_{n\in \Z}\alpha^n(Q)$, then $\alpha^n(g)\in Q$ for each
$n\in \Z$, whence $\alpha^\Z(g)$ is relatively compact and thus
$g\in M_\alpha$. Hence $I\sub M_\alpha$. Since~$I$ is
$\alpha$-stable and $I\sub Q\cap M_\alpha=P$, we deduce that
$I=\bigcap_{n\in \Z}\alpha^n(I) \sub\bigcap_{n\in
\Z}\alpha^n(P)=\{1\}$. Thus $\bigcap_{n\in\Z}\alpha^n(Q)=\{1\}$ and
thus~$\alpha$ is expansive.

(b) If $\alpha$ is expansive, then $U_\alpha U_{\alpha^{-1}}$
is open by Lemma~\ref{basicexpa}\,(d).
If $U_\alpha$ is closed, then the set $U_\alpha M_\alpha U_{\alpha^{-1}}$
is open in~$G$ and the product map
\[
U_\alpha \times M_\alpha \times U_{\alpha^{-1}}\to
U_\alpha M_\alpha U_{\alpha^{-1}}, \quad (x,y,z)\mto xyz
\]
is a homeomorphism (by \ref{recall1}\,(e)
above and part (f)
from the theorem in~\cite{Glo}).
Hence $U_\alpha U_{\alpha^{-1}}\cap M_\alpha=\{1\}$.
If $U_\alpha U_{\alpha^{-1}}$ is open,
this implies that $M_\alpha$ is discrete.
If $M_\alpha$ is discrete, then
$\alpha|_{M_\alpha}$ is expansive and hence
also~$\alpha$, by~(a).
\end{proof}
\begin{numba}\label{cpconnorm}
If $\alpha$ is an automorphism of a totally disconnected,
\emph{compact} group~$G$, then
$U_\alpha$ and $U_{\alpha^{-1}}$ are
normal in~$G$ (since $G=M_\alpha=M_{\alpha^{-1}}$ in \ref{recall1}\,(f)),
and hence so is the nub $U_0=\wb{U_\alpha}\cap \wb{U_{\alpha^{-1}}}$.
\end{numba}
\begin{numba}\label{torsionclosure}
Recall that a group~$G$ is called a torsion group \emph{of finite
exponent} if there exists $m\in \N$ such that $g^m=1$ for all $g\in
G$. If such a group is a subgroup of a topological group~$H$, then
also $g^m=1$ for all $g$ in the closure $\wb{H}$ of $G$ in~$H$, and
thus also $\wb{H}$ is a torsion group of finite exponent.
\end{numba}
\begin{numba}\label{conlarge}
Let $\alpha$ be an expansive automorphism of a totally disconnected,
compact group~$G$. Then the following holds:
\begin{itemize}
\item[(a)]
The nub $U_0$ is open in~$G$ (see \cite[Lemma~5.1]{Wi3}), whence
also $\wb{U_{\alpha}}$ and $\wb{U_{\alpha^{-1}}}$ are open in~$G$
(by \ref{recall1}\,(b)).
\item[(b)]
$G$ is a torsion group of finite exponent [By (a) and
\ref{cpconnorm}, $U_0$ is a normal subgroup and $G/U_0$ is finite,
hence a torsion group of finite exponent. It therefore suffices to
show that~$U_0$ is a torsion group of finite exponent. This is
immediate from
\cite[Proposition~4.4, Theorem 6.2 and Proposition 6.3]{Wi3}).
\item [(c)] Since $U_\alpha $ and $U_{\alpha ^{-1}}$ are normal in $G$,
the set $U_\alpha U_{\alpha^{-1}}$ is an open $\alpha$-stable subgroup
of~$\!G\!$ contained in~$\!U_0$. \hspace*{-1mm}Thus $\!U_0 \!=\!U_\alpha U_{\alpha ^{-1}}\!$
by \hspace*{-.3mm}\cite[Corollary 4.3]{Wi3}.
\end{itemize}
\end{numba}
\begin{la}\label{nochange}
Let $\alpha$ be an automorphism of a totally disconnected,
locally compact group~$G$ and $U_0$ be the nub of~$\alpha$.
Let $H\sub G$ be a closed, $\alpha$-stable subgroup.
Then the following holds:
\begin{itemize}
\item[\rm(a)]
Then nub of $\alpha|_H$ is contained in~$U_0$.
\item[\rm(b)]
If $U_0\sub H$, then the nub of $\alpha|_H$
coincides with~$U_0$.
\end{itemize}
\end{la}
\begin{proof}
(a) Let~$W$ be the nub of~$\alpha|_H$.
Using \ref{recall1}\,(b) twice, we deduce that
$W=
\wb{U_{\alpha|_H}} \cap \wb{U_{(\alpha|_H)^{-1}}}
\sub
\wb{U_\alpha} \cap \wb{U_{\alpha^{-1}}}=U_0$.

(b) Since $U_0\sub H$, we have
$U_\alpha\cap U_0\sub
U_\alpha\cap H=U_{\alpha|_H}$
and thus $U_0=\wb{U_0\cap U_\alpha}
\sub \wb{U_{\alpha|_H}}$ (using \ref{recall1}\,(c)).
Likewise, $U_0\sub \wb{U_{(\alpha|_H)^{-1}}}$.
Hence $U_0\sub \wb{U_{\alpha|_H}}\cap \wb{U_{(\alpha|_H)^{-1}}}=W$.
Since $W\sub U_0$ by (a), equality follows.
\end{proof}
We shall also need certain facts concerning
contractive automorphisms.
\begin{numba}\label{recall2}
Let $\alpha$ be a contractive automorphism of a topological
group $G$.
\begin{itemize}
\item[(a)]
If $G\not=\{1\}$, then $G$ is infinite and
non-discrete. [If $x\in G\setminus\{1\}$,
then $\alpha^n(x)\not=1$ for all
$n$ and $\alpha^n(x)\to 1$, entailing
that the topological group $G$ is not discrete
and hence infinite.]
\item[(b)]
If $G$ is locally compact,
then $\alpha$ is \emph{compactly contractive},
i.e., for each identity neighbourhood
$U\sub G$ and compact set $K\sub G$ there exists $m\in \N$
such that $\alpha^n(K)\sub U$ for all $n\geq m$
(see \cite[Lemma~1.4\,(iv)]{Sie}).
Moreover, $G$ is non-compact (unless $G=\{1\}$);
see \cite[3.1]{Sie}.
\end{itemize}
\end{numba}
\begin{numba}\label{recall3}
Let $\alpha$ be a contractive automorphism
of a totally disconnected,
locally compact group $G$. Then the following holds:
\begin{itemize}
\item[(a)]
If $G\not=\{1\}$,
then $\Delta_G(\alpha^{-1})$ is an integer $\geq 2$
(see \cite[Proposition~1.1\,(e)]{GaW}).
\item[(b)] If
$G=G_0\supset G_1\supset G_2\supset\cdots\supset G_n=\{1\}$
is a series of $\alpha$-stable closed subgroups
of $G$ such that $G_j$ is a proper normal
subgroup of $G_{j-1}$ for all $j\in\{1,\ldots, n\}$,
then~$n$ is bounded by the number of prime factors
of $\Delta_G(\alpha^{-1})$
(see \cite[Lemma~3.5]{GaW}).
\item[(c)]
$G=\dv(G)\times \tor(G)$ as a topological group
for a divisible torsion free group $\dv(G)$
and a torsion group $\tor(G)$ of finite
exponent (cf.\ \cite[Theorem~B]{GaW}).
\item[(d)]
If $N\sub G$ is an $\alpha$-stable closed normal subgroup, $q\colon
G\to G/N$ the canonical quotient morphism and $\bar{\alpha}$ the
automorphism of $G/N$ induced by~$\alpha$, then $q(\dv(G))=\dv(G/N)$
and $q(\tor(G))=\tor(G/N)$. [The inclusions $q(\dv(G))\sub \dv(G/N)$
$q(\tor(G))\sub\tor(G/N)$ are clear.  Since
$G/N=q(\dv(G)\tor(G))=q(\dv(G))q(\tor(G))$ and $G/N=\dv(G/N)\times
\tor(G/N)$, equality follows.]
\end{itemize}
\end{numba}
\begin{rem}\label{contrexpa}
If an automorphism $\alpha\colon G\to G$ of a totally disconnected,
locally compact group is contractive, then it is also expansive. To
see this, let $V\sub G$ be any compact neighbourhood of the identity. If
$g\in G\setminus\{1\}$, then $G\setminus\{g\}$ is an identity
neighbourhood, whence there exists $n\in \N$ such that
$\alpha^n(V)\sub G\setminus\{g\}$ (see \ref{recall2}\,(b)). Thus
$g\not\in \alpha^n(V)$ and we have shown that $\bigcap_{n\in
\Z}\alpha^n(V)=\{1\}$.
\end{rem}
\section{Making contraction groups locally compact}\label{secrefine}
The problem of refining group topologies
on contraction groups was studied by Siebert~\cite{Si2}.
The following special case is
useful for our purposes.
\begin{defn}\label{makelcp}
Let $G$ be a topological group, with topology $\tau$, and
\[
\alpha\colon (G,\tau)\to (G,\tau)
\]
be a contractive automorphism.
We say that $(G,\alpha)$ (or simply $G$) \emph{can be made locally compact}
if there exists a locally compact topology $\tau^*$
on $G$ making it a topological group
such that $\tau\sub \tau^*$ and $\alpha\colon (G,\tau^*)\to (G,\tau^*)$
is a contractive automorphism.
The topology $\tau^*$ is unique
(if it exists), as will be recalled
presently.
We write $G^*$ for $G$, endowed with the topology $\tau^*$.
If $\tau$ is totally disconnected,
then also $\tau^*$ is totally disconnected (as the inclusion map $G^*\to G$
is continuous).
\end{defn}
The following fact is also obtained in \cite[Corollary~8]{Si2}
as part of a more general theory.
\begin{la}\label{unifiner}
$\tau^*$ is uniquely determined by the properties
from Definition~\emph{\ref{makelcp}}.
\end{la}
\begin{proof}
Assume that $\hat{\tau}$ is a group topology on~$G$
with the same properties as~$\tau^*$.
We show that the identity map
\[
\phi\colon (G,\hat{\tau})\to (G,\tau^*),\quad x\mto x
\]
is continuous. Reversing the roles of $\hat{\tau}$ and $\tau^*$,
also $\phi^{-1}$ will be continuous and thus $\hat{\tau}=\tau^*$.
Because $\phi$ is a homomorphism, we need only prove its
continuity at~$1$.
By local compactness, there exists a compact
identity neighbourhood
$V\sub (G,\hat{\tau})$.
After replacing~$V$ with the closure
of its interior~$V^0$, we may assume that $V^0$ is dense in~$V$.
Then~$V$ is also compact in $(G,\tau)$.
Let $U\sub (G,\tau^*)$ be an arbitrary identity neighbourhood,
and $W\sub (G,\tau^*)$ be a compact identity neighbourhood
such that $WW^{-1}\sub U$.
Then~$W$ is also compact in $(G,\tau)$, hence closed in $(G,\tau)$
and hence closed in $(G,\hat{\tau})$.
Since $\alpha\colon (G,\tau^*)\to(G,\tau^*)$
is contractive, we have $G=\bigcup_{n\in \N_0}\alpha^{-n}(W)$.
Hence~$V$ is the countable union of the closed subsets
$V\cap \alpha^{-n}(W)$, for $n\in \N_0$.
By the Baire Category Theorem,
there exists $m\in \N_0$ such that
$V\cap \alpha^{-m}(W)$ has non-empty interior in~$V$.
Because $V^0$ is dense in~$V$, we deduce
that $\alpha^{-n}(W)\cap V^0$ has non-empty interior in~$V^0$,
and so~$W$ has non-empty interior~$W^0$ in $(G,\hat{\tau})$.
Then $W^0(W^0)^{-1}$ is an identity neighbourhood in $(G,\hat{\tau})$
and hence~$U$ is an identity neighbourhood in $(G,\hat{\tau})$.
Since $U=\phi^{-1}(U)$, we see that $\phi$ is continuous at~$1$.
\end{proof}
See also \cite[Proposition~9]{Si2} for the following fact.
\begin{la}\label{expalcp}
Let $G$ be a totally disconnected, locally compact
group. If an automorphism $\alpha\colon G\to G$ is expansive,
then $(U_\alpha,\alpha|_{U_\alpha})$
and $(U_{\alpha^{-1}},\alpha^{-1}|_{U_{\alpha^{-1}}})$
can be made locally compact.
\end{la}
\begin{proof}
Let $V\sub G$ be a compact open subgroup
such that $\bigcap_{n\in \Z}\alpha^n(V)=\{1\}$.
Then
$\alpha^n(V_-)=V_-\cap\bigcap_{k=1}^n\alpha^k(V)$
is open in~$V_-$ for each
$n\in \N_0$.
Since $V_-$ is compact and
$\bigcap_{n\in \N_0}\alpha^n(V_-)
=\bigcap_{n\in\Z}\alpha^n(V)=\{1\}$,
the open subgroups $\alpha^n(V_-)$ form
a basis of identity neighbourhoods in~$V_-$,
for $n\in \N_0$.
If $n\in \N$ and
$g\in U_\alpha=V_{--}=\bigcup_{k\in \N_0}\alpha^{-k}(V_-)$
(see Proposition~\ref{basicexpa}\,(c)),
then $g\in \alpha^{-k}(V_-)$ for some $k\in \N_0$
and $\alpha^n(V_-)$ is an open subgroup
of the topological group $\alpha^{-k}(V_-)$.
Hence, there exists $m\in \N_0$
such that $g\alpha^m(V_-)g^{-1}\sub \alpha^n(V_-)$.
By the preceding, there exists a group
topology~$\tau^*$ on~$U_\alpha$
for which $\{\alpha^n(V_-)\colon n\in \N_0\}$
is a basis of identity neighbourhoods.
Thus~$V_-$ is an open subgroup of $(U_\alpha,\tau^*)$,
and the latter group induces the given compact
topology on~$V_-$
(as $\{\alpha^n(V_-)\colon n\in \N_0\}$ is a basis of identity neighbourhoods
for both topologies on~$V_-$).
Thus $(U_\alpha,\tau^*)$ is a totally disconnected,
locally compact group and $\alpha$ is still
continuous (being continuous on the open subgroup $V_-$)
as well as $\alpha^{-1}$ (being continuous on the
subgroup $\alpha(V_-)$
which is open in $V_-$).
Since $U_\alpha=\bigcup_{n\in \N_0}\alpha^{-n}(V_-)$
and $(\alpha^n(V_-))_{n\in \N_0}$ is a basis
of identity neighbourhoods,
it readily follows that the automorphism
$\alpha$ of $(U_\alpha,\tau^*)$ is contractive.
\end{proof}
\begin{rem}
$U_\alpha$ can also be made locally compact
if~$\alpha$ is an arbitrary (not necessarily expansive)
analytic automorphism of a Lie group~$G$
over a local field
(see Proposition 13.3\,(b)
in the extended preprint version of~\cite{GIM}).
\end{rem}
\begin{example}
The right-shift $\alpha$ is an automorphism
of the compact group $G:=(\Z_p)^\Z$,
where~$\Z_p$ is the additive group of $p$-adic integers.
The contraction group~$U_\alpha$ is non-trivial,
as it is the group of all $(z_n)_{n\in \Z}$ such that $z_n\to 0$
as $n\to {-\infty}$.
Then $U_\alpha$ cannot be locally compact,
because $\tor(U_\alpha)\sub \tor(G)=\{0\}$
and $\dv(U_\alpha)\sub \dv(G)=\{0\}$.
Thus, if~$U_\alpha$
could be made locally compact,
then we would have $U_\alpha=\dv(U_\alpha)+\tor(U_\alpha)=\{0\}$,
contradicting $U_\alpha\not=\{0\}$.
\end{example}
\begin{la}\label{subandquot}
Let $G$ be a topological group and $\alpha\colon G\to G$
a contractive automorphism such that $(G,\alpha)$
can be made locally compact.
Then we have:
\begin{itemize}
\item[\rm (a)]
$(H,\alpha|_H)$ can be made locally compact
for each closed $\alpha$-stable subgroup~$H$ of~$G$
\emph{(}or $G^*$\emph{)},
and~$H^*$ carries the topology induced
by~$G^*$.
\item[\rm (b)]
If $\phi\colon G\to H$ is a continuous
homomorphism to a topological group $H$ admitting
an automorphism $\beta\colon H\to H$
such that $\beta\circ \phi=\phi\circ \alpha$,
then $\beta|_{\phi(G)}$ is contractive,
$\phi(G)$ can be made locally compact
and
\[
G^*\to \phi(G)^*, \quad x\mto \phi(x)
\]
is a topological quotient map.
\end{itemize}
\end{la}
\begin{proof}
(a) Let $\tau$ and $\tau^*$ be as in Definition~\ref{makelcp}.
Since $\tau\sub \tau^*$, $H$ is closed in~$G^*$
(in either case)
and hence~$H$ is a locally compact group
in the topology~$\sigma$
on~$H$ induced by $G^*$, which is finer than the topology
induced by $G$ and turns $\alpha|_H$ into a contractive
automorphism of $(H,\sigma)$. Thus $H^*=(H,\sigma)$.

(b) Let $\sigma$ be the topology on~$\phi(G)$ turning $G^*\to
(\phi(G),\sigma)$, $x\mto\phi(x)$ into a quotient map. Then $\sigma$
is finer than the topology induced on $\phi(G)$ by~$H$. Moreover,
$(\phi(G),\sigma)\cong G^*/\ker\phi$ is locally compact. Since
$\beta|_{\phi(G)}\circ \phi=\phi\circ \alpha\colon G^*\to (\phi(G),\sigma)$
is continuous, the map $\beta|_{\phi(G)}$ is continuous with
respect to the quotient topology~$\sigma$. Also
$(\beta|_{\phi(G)})^{-1}$ is continuous, by an analogous argument.
Finally, $\beta|_{\phi(G)}\colon (\phi(G),\sigma)\to
(\phi(G),\sigma)$ is contractive: Since
$\beta^n\circ\phi=\phi\circ\alpha^n$, this follows from the facts
that $\alpha\colon G^*\to G^*$ is contractive and $\phi\colon G^*\to
\phi(G)^*$ is continuous.
\end{proof}
The following observation is crucial
for many of our arguments.
\begin{prop}\label{almostmax}
Let $\alpha$ be an expansive automorphism
of a totally disconnected, locally compact group~$G$
and
\[
G=G_0\supseteq G_1\supseteq\cdots\supseteq G_n
\]
be $\alpha$-stable closed subgroups of~$G$ such that
$G_j$ is normal in $G_{j-1}$ for all $j\in \{1,\ldots,n\}$.
Let $J$ be the set of all
$j\in\{1,\ldots, n\}$ such that
$G_{j-1}/G_j$ is not discrete.
Then
\[
\# J\leq \ell_\alpha +\ell_{\alpha^{-1}},
\]
where $\ell_\alpha$ is the number of prime
factors of
$\Delta_{U_\alpha^*}(\alpha^{-1}|_{U_\alpha^*})$
and $\ell_{\alpha^{-1}}$ is the
number of prime factors
of $\Delta_{U_{\alpha^{-1}}^*}(\alpha|_{U_{\alpha^{-1}}^*})$.
\end{prop}
\begin{proof}
Let $J_\alpha$ (resp., $J_{\alpha^{-1}}$) be the set of all
$j\in\{1,\ldots, n\}$ such that $U_\alpha\cap G_j \subset
U_\alpha\cap G_{j-1}$ (resp., $U_{\alpha^{-1}}\cap G_j \subset
U_{\alpha^{-1}}\cap G_{j-1}$). If $j\in \{1,\ldots, n\}\setminus
(J_\alpha\cup J_{\alpha^{-1}})$, then $U_\alpha\cap G_j
=U_\alpha\cap G_{j-1}$ and $U_{\alpha^{-1}}\cap G_j =
U_{\alpha^{-1}}\cap G_{j-1}$. Since $\alpha$ is expansive,
$(U_\alpha\cap G_{j-1})(U_{\alpha^{-1}}\cap G_{j-1})$ is open
in~$G_{j-1}$ (see \ref{basicexpa} (d) and (e)). We deduce that~$G_j$
is open in~$G_{j-1}$ and thus $j\not\in J$. Hence $J\sub
J_\alpha\cup J_{\alpha^{-1}}$, whence $\# J \leq \# J_\alpha + \#
J_{\alpha^{-1}}\leq \ell_\alpha+\ell_{\alpha^{-1}}$ (using
\ref{recall3}\,(b) in the last step).
\end{proof}
We shall use a simple fact.
\begin{la}\label{compdiv}.
Let $K$ be a compact group and $D\sub K$ be a subgroup which is
divisible. Then also the closure $\wb{D}$ is divisible. If~$K$ is
totally disconnected, then $D=\{1\}$.
\end{la}
\begin{proof}
For each $m\in \N$, the map $f_m\colon \wb{D}\to\wb{D}$, $g\mto g^m$
is continuous and hence has compact image.
As the image contains~$D$ by hypothesis, we see that
$f_m(\wb{D})=\wb{D}$. Thus $\wb{D}$ is divisible.\\[3mm]
If $K$ is totally disconnected, then $K$ is a pro-finite group. In
particular, the homomorphisms $f\colon K\to F$ to finite groups $F$
separate points on~$\wb{D}$. But each $f(\wb{D})$ is both finite and
divisible and therefore the trivial group. Hence also
$\wb{D}=\{1\}$.
\end{proof}
\begin{la}\label{torsionfacts}
Let $\alpha$ be an automorphism of
a totally disconnected, locally compact
group~$G$. If $U_\alpha$ can be made locally compact
$($e.g., if $\alpha$ is expansive$)$,
then the following holds:
\begin{itemize}
\item[\rm(a)]
$U_\alpha\cap U_0=T_\alpha\cap U_0$;
\item[\rm(b)]
$U_0=\wb{T_\alpha\cap U_0}$;
\item[\rm(c)]
$\wb{T_\alpha}\cap U_\alpha=T_\alpha$;
\item[\rm(d)]
$\wb{T_\alpha}=T_\alpha U_0$.
\end{itemize}
If both $U_\alpha$ and $U_{\alpha^{-1}}$
can be made locally compact, then we also have:
\begin{itemize}
\item[\rm(e)]
$U_0=\wb{T_\alpha}\cap \wb{T_{\alpha^{-1}}}$.
\end{itemize}
\end{la}
\begin{proof}
(a) By Lemma~\ref{subandquot}\,(a),
$U_\alpha\cap U_0$ can be made locally compact.
Thus~$\alpha$ restricts to a contractive automorphism~$\beta$
of $(U_\alpha\cap U_0)^*$,
enabling us to write
$(U_\alpha\cap U_0)^*=D_\beta T_\beta$.
Since~$U_0$ is compact and totally disconnected,
its divisible subgroup~$D_\beta$ has to be trivial,
by Lemma~\ref{compdiv}.
Thus $U_\alpha\cap U_0=U_\beta=T_\beta\sub T_\alpha$
and hence $U_\alpha\cap U_0=T_\alpha\cap U_0$.

(b) Since $U_0=\wb{U_\alpha\cap U_0}$ (see \ref{recall1}\,(c)),
the assertion is immediate from (a).

(c) $T_\alpha\sub \wb{T_\alpha}\cap U_\alpha$ is trivial.
Because $T_\alpha$ is a torsion group
of finite exponent (see \ref{recall3}\,(c)), also $\wb{T_\alpha}$
is a torsion group, see \ref{torsionclosure}.
Let~$\beta$ the restriction of $\alpha$
to the closed $\alpha$-stable subgroup
$\wb{T_\alpha}\cap U_\alpha^*$
of $U_\alpha^*$.
Then $\wb{T_\alpha}\cap U_\alpha
=D_\beta T_\beta$.
Since $D_\beta$ is torsion-free (see \ref{recall3}\,(c))
and $\wb{T_\alpha}$ a torsion group, $D_\beta=\{1\}$
follows. Thus $\wb{T_\alpha}\cap U_\alpha=T_\beta\sub T_\alpha$.

(d) We have $\wb{T_\alpha}\sub \wb{U_\alpha} = U_\alpha U_0$ (see
1.2 (a)). Since $U_0\sub \wb{T_\alpha}$ by~(b), we deduce that
$\wb{T_\alpha}=(U_\alpha \cap \wb{T_\alpha})U_0=T_\alpha U_0$
(using~(c) for the last equality).

(e) We have $U_0=\wb{U_\alpha}\cap \wb{U_{\alpha^{-1}}}\supseteq
\wb{T_\alpha}\cap \wb{T_{\alpha^{-1}}}$ (see \ref{recall1}\,(b))
and $U_0=\wb{T_\alpha\cap U_0}\cap \wb{T_{\alpha^{-1}}\cap U_0}
\sub \wb{T_\alpha}\cap \wb{T_{\alpha^{-1}}}$
(using~(b)).
\end{proof}
\begin{la}\label{sigmacomp}
Let $\alpha$ be an automorphism of a locally compact group~$G$ and
$H\sub G$ be an $\alpha$-stable subgroup such that $\alpha|_H$ is
contractive. Then the closure $\wb{H}\sub G$ is $\sigma$-compact.
\end{la}
\begin{proof}
Let $K\sub \wb{H}$
be a compact identity neighbourhood.
Then $\bigcup_{n\in \Z}\alpha^n(K)$ is a $\sigma$-compact
subset of $\wb{H}$ and generates
a $\sigma$-compact subgroup~$S$ of~$\wb{H}$.
Since~$S$ is an $\alpha$-stable open subgroup of~$\wb{H}$
and $\alpha|_H$ is contractive,
we have $H\sub S$ and thus $S=\wb{H}$
(since~$S$ is closed). Hence $\wb{H}=S$
is $\sigma$-compact.
\end{proof}
\section{Proof of Theorem A}\label{autquot}
Let $G$ be a totally disconnected, locally compact group, $\alpha$
be an expansive automorphism of $G$ and $N\sub G$ be an
$\alpha$-stable, closed normal subgroup. Let $q\colon G\to G/N$ be
the canonical quotient morphism and $\wb{\alpha}$ be the
automorphism of $G/N$ induced by~$\alpha$ (determined by
$\wb{\alpha}\circ q=q\circ\alpha$). Then $\alpha|_N$ is expansive,
by Lemma~\ref{basicexpa}\,(e). We show that also $\wb{\alpha}$ is
expansive. By Proposition~\ref{givesexp}\,(a), we need only show
that $\wb{\alpha}$ restricts to an expansive automorphism
of~$M_{\wb{\alpha}}$. After replacing~$G$ with $q^{-1}(M_{\wb{\alpha}})$
(which is closed since~$M_{\wb{\alpha}}$ is closed), we may assume
that $G/N=M_{\wb{\alpha}}$. Let~$U$ be a subgroup of $G/N$ tidy
for~$\wb{\alpha}$. Then $U=U_+=U_-$ as $\wb{\alpha}^\Z(g)$ is
relatively compact for each $g\in U$ (cf.\ \cite[Lemma~9]{Wi1}) and
thus~$U$ is an $\wb{\alpha}$-stable, compact open subgroup of~$G/N$.
After replacing $G$ with~$q^{-1}(U)$, we may assume that~$G/N$ is
compact. Using \cite[Proposition~5.1]{Wi3} and the metrizability of
$G/N$, we find a descending sequence $(H_n)_{n\in\N}$ of
$\wb{\alpha}$-stable closed normal subgroups~$H_n$ of $G/N$ such
that~$\wb{\alpha}$ induces an expansive automorphism~$\alpha_n$ on
$(G/N)/H_n$ for each $n\in \N$ and $G/N$ is the projective limit
$G/N=\pl\, (G/N)/H_n$.\vspace{-.3mm} Set $L_n:=q^{-1}(H_n)$; then
$(L_n)_{n\in\N}$ is a descending sequence of $\alpha$-stable closed
normal subgroups
of~$G$, with $\bigcap_{n\in\N}L_n=N$.\\[3mm]
\emph{There exists $m\in\N$ such that
$L_n$ is open in $L_m$ for all $n\geq m$.}
Indeed, if this was wrong, we could find
a subsequence $(L_{n_k})_{k\in\N}$
such that, for each $k\in\N$, the normal subgroup
$L_{n_{k+1}}$ is not open in $L_{n_k}$.
This contradicts Proposition~\ref{almostmax}.\\[3mm]
After passing to a subsequence, we may assume that~$L_n$ is open
in~$L_1$ for each $n\in\N$. Hence $L_n$ contains both $U_\alpha\cap
L_1$ and $U_{\alpha^{-1}}\cap L_1$. As a consequence,
$N=\bigcap_{n\in\N}L_n$ contains both $U_\alpha\cap L_1$ and
$U_{\alpha^{-1}}\cap L_1$. Hence~$N$ is open in~$L_1$ (and in each
$L_n$), using that $\alpha|_{L_1}$ is expansive and thus
$(U_\alpha\cap L_1)(U_{\alpha^{-1}}\cap L_1)$ an open subset of~$L_1$ (see
\ref{basicexpa} (d) and (e)). This implies that the compact group
$H_1\cong L_1/N$ is discrete and hence a finite group. Since
$H_1\supseteq H_2\supseteq\cdots$ with $\bigcap_{n\in\N}H_n=\{1\}$,
we deduce that $H_n=\{1\}$ for some~$n$. Since~$\bar{\alpha}$
corresponds to the expansive automorphism~$\alpha_n$ on
$(G/N)/H_n\cong G/N$, we see that~$\bar{\alpha}$ is expansive.\\[3mm]
Conversely, assume that both $\alpha|_N$ and $\wb{\alpha}\colon G/N\to G/N$
are expansive.\footnote{Compare also \cite[Proposition~6.1]{Wi3}.
The compactness of~$G$ assumed there is inessential
for this part of the proof of~\cite[Proposition~6.1]{Wi3}.}
Then there is an open identity neighbourhood $P\sub G/N$
such that $\bigcap_{n\in \Z}\wb{\alpha}^n(P)=\{1\}$,
and an open identity neighbourhood $Q\sub N$ such that
$\bigcap_{n\in \Z}\alpha^n(Q)=\{1\}$.
After shrinking~$Q$, we may assume that $Q\sub q^{-1}(P)$.
Then $Q= N\cap V$ for some open identity neighbourhood
$V\sub G$. After replacing~$V$ with $V\cap q^{-1}(P)$,
we may assume that $V\sub q^{-1}(P)$.
We have $N=q^{-1}(\bigcap_{n\in \Z}\wb{\alpha}^n(P))
=\bigcap_{n\in \Z}\alpha^n(q^{-1}(P))$, entailing
that $I:=\bigcap_{n\in \Z}\alpha^n(V)$ is an
$\alpha$-stable subset of~$N$.
Since $I=I\cap N\sub V\cap N=Q$,
we deduce that $I=\bigcap_{n\in\Z}\alpha^n(I)
\sub\bigcap_{n\in\Z}\alpha^n(Q)=\{1\}$.
Hence $I=\{1\}$ and~$\alpha$ is expansive.\Punkt
\section{Proof of Theorem B}
The following lemma is useful.
\begin{la}\label{twocases}
Let $\alpha$ be an expansive automorphism
of a totally disconnected, locally compact group~$G$
such that every $\alpha$-stable, closed normal
subgroup $N\sub G$ is open or discrete.
Let $C\sub G$ be the topologically characteristic
subgroup\footnote{Thus $C$ is the subgroup
generated by $\bigcup_{\beta\in \Aut(G)}\beta(U_\alpha\cup U_{\alpha^{-1}})$.}
of~$G$ generated by $U_\alpha\cup U_{\alpha^{-1}}$.
Let $\wb{[C,C]}$ be the closure of the commutator group of~$C$.
Then~$C$ is an open normal subgroup of~$G$,
and one of the following cases occurs:
\begin{itemize}
\item[\rm(a)]
$\wb{[C,C]}$
is open in~$G$,
in which case
$C=\wb{[C,C]}$
is topologically perfect.
\item[\rm(b)]
$\wb{[C,C]}$ is discrete.
\end{itemize}
\end{la}
\begin{proof}
$C$ is an open subgroup of~$G$ as it
contains the open set $U_\alpha U_{\alpha^{-1}}$ (see \ref{basicexpa}\,(d)).
The closed subgroup $\wb{[C,C]}$ is topologically characteristic
in~$C$, whence it is topologically characteristic
in~$G$ and hence $\alpha$-stable and normal.
Therefore $\wb{[C,C]}$ is open or discrete.
If $\wb{[C,C]}$ is open, then it contains
$U_\alpha\cup U_{\alpha^{-1}}$.
Hence $\wb{[C,C]}=C$, using that~$C$
is the smallest topologically characteristic subgroup of~$G$
which contains $U_\alpha\cup U_{\alpha^{-1}}$.
\end{proof}
\noindent{\bf Proof of Theorem B.}
Define $\ell_{\alpha}$ and $\ell_{\alpha^{-1}}$ as in
Proposition~\ref{almostmax}.
For every series
\[
\Sigma\colon G=G_0\rhd G_1\rhd\cdots\rhd G_n=\{1\}
\]
of $\alpha$-stable closed subgroups of~$G$, let
$J_\Sigma$ be the set
of all $j\in \{1,\ldots,n\}$ such that $G_{j-1}/G_j$
is not discrete.
Then $J_\Sigma\leq
\ell_{\alpha}+\ell_{\alpha^{-1}}$,
by Proposition~\ref{almostmax},
entailing that the maximum
\[
m:=\max_\Sigma J_\Sigma
\]
over all series $\Sigma$
exists. Let $\Sigma\colon G=G_0\rhd G_1\rhd\cdots\rhd G_n=\{1\}$
be a series with $J_\Sigma=m$.
Let $N\sub G_{j-1}$ be an $\alpha$-stable closed
normal subgroup with $G_j\sub N$.
If neither $G_{j-1}/N$ nor $N/G_j$ was discrete,
we would have $J_{\Sigma\cup \{N\}}=J_\Sigma+1$,
a contradiction. Thus~$N$ will be open in~$G_{j-1}$
or~$G_j$ open in~$N$.\\[3mm]
For $j\in J_\Sigma$, let $q_j\colon G_{j-1}\to G_{j-1}/G_j$
be the canonical quotient morphism,\linebreak
$\alpha_j\colon
G_{j-1}/G_j \to G_{j-1}/G_j$ be the automorphism
induced by~$\alpha$
and $C_j\subseteq G_{j-1}/G_j$
be the topologically characteristic
subgroup generated by $U_{\alpha_j}\cup U_{\alpha_j^{-1}}$.
If $\wb{[C_j,C_j]}$ is open in $G_{j-1}/G_j$,
define $M_j:=N_j:=q_j^{-1}(C_j)$;
thus $G_{j-1}/M_j\cong (G_{j-1}/G_j)/C_j$
and $M_j/N_j=\{1\}$ are discrete
and $N_j/G_j\cong C_j$ is topologically perfect.
If $\wb{[C_j,C_j]}$ is discrete, we
define $M_j:=q_j^{-1}(C_j)$
and $N_j:=q_j^{-1}(\wb{[C_j,C_j]})$
(by Lemma~\ref{twocases}, only these two cases
can occur).
Then $G_{j-1}/M_j\cong (G_{j-1}/G_j)/C_j$
is discrete, $M_j/N_j\cong C_j/\wb{[C_j,C_j]}$
is abelian and $N_j/G_j\cong \wb{[C_j,C_j]}$
is discrete.
Hence
\[
\Sigma':=\Sigma\cup \bigcup_{j\in J_\Sigma}\{M_j,N_j\}
\]
is a series of $\alpha$-stable closed subnormal
subgroups such that all non-discrete subfactors
are abelian or topologically perfect.
Since $\# J_{\Sigma'}=\# J_\Sigma$ is maximal,
all non-discrete subfactors of $\Sigma'$
have the property that all stable closed
normal subgroups are open or discrete.\Punkt\vspace{3mm}

\noindent
Using the recent theory of elementary groups~\cite{Wes},
slightly more detailed information on the factor groups can be obtained,
in the case of second countable groups.
Recall that the class of
elementary groups is the smallest class of totally disconnected,
second
countable, locally compact groups that contains all countable discrete groups and all second countable pro-finite groups,
and is closed under extensions as well as countable increasing unions. A totally disconnected, second countable, locally compact group $G$ is called \emph{elementary-free}
if all of its elementary closed normal subgroups and all of its
elementary Hausdorff quotient groups are trivial \cite[Definition~7.14]{Wes}.
If $\alpha$ is an expansive automorphism of a totally disconnected, locally compact non-trivial group~$G$
and $G$ does not have closed $\alpha$-stable subgroups
except for $G$ and $\{1\}$, then $(G,\alpha)$ is called a \emph{simple expansion group}.
Note that if $G$ is a non-trivial elementary-free group and $\alpha$ an expansive automorphism
of~$G$ such that every $\alpha$-stable closed normal subgroup of $G$ is discrete
or open, then $(G,\alpha)$ is a simple expansion group. We remark:
\begin{rem}\label{newremark}
If $G$ is second countable in Theorem B, then one can achieve there that each of the quotient groups $G_{j-1}/G_j$ is discrete, abelian,
both topologically perfect and elementary, or an elementary-free simple expansion group.\\[2.3mm]
In fact, let us consider a topologically perfect factor $Q:=G_{j-1}/G_j$ in a series
all of whose factors are discrete, abelian, or topologically perfect, and which has a maximum number
of non-discrete factors.
By \cite[Theorem 7.15]{Wes}, there are topologically characteristic, closed subgroups
$D_1$ and $D_2$ of $Q$ such that $Q \supseteq D_1 \supseteq D_2 \supseteq \{1\}$
and, moreover, both $D_2$ and $Q/D_1$ are elementary and $D_1/D_2$ is elementary-free.
Let $N_1$ and $N_2$ be the pre-images of $D_1$ and $D_2$, respectively,
under the quotient morphism $G_{j-1}\to G_{j-1}/G_j$. Then $N_1$ and $N_2$ are $\alpha$-stable
closed normal subgroups of $G_{j-1}$, and $G_{j-1}\rhd N_1\rhd  N_2\rhd  G_j$.

Case 1:
If $D_1/D_2$ is non-trivial, then the elementary-free group $N_1/N_2\cong D_1/D_2$ is non-discrete (as it would be elementary otherwise),
whence $N_2$ is not open in $N_1$. Hence $N_2$ is not open in $G_{j-1}$
and hence $N_2/G_j$ is discrete (by maximailty of the number of non-discrete factors).
Again by maximality,
$G_{j-1}/N_1$ is discrete and $N_1/N_2$ does not have closed normal subgroups
stable under the induced expansive automorphism other than open or discrete subgroups.
So, the elementary-free group $N_1/N_2$ is a simple expansion group.

Case 2: If $D_1/D_2$ is trivial, then the topologically perfect group $G_{j-1}/G_j=Q\rhd D_1=D_2\rhd \{1\}$ is
elementary, as any extension of elementary groups. 
\end{rem}
\section{Proof of Theorem C}
Since $U_0\sub \wb{U_\alpha}$,
we may replace~$G$ with $\wb{U_\alpha}$
without changing the nub (see~\ref{nochange}), or
$\wb{T_\alpha}$, or $D_\alpha$.
We may therefore assume that~$U_\alpha$ is dense in~$G$.
Since~$U_0$ normalizes~$U_\alpha$ (see \ref{recall1}\,(f)),
$U_\alpha$ is a normal subgroup
of $G=U_\alpha U_0$ (exploiting \ref{recall1}\,(a)).
Hence also the characteristic subgroups~$D_\alpha$ and~$T_\alpha$
of~$U_\alpha$ are normal in~$G$.
Therefore also $\wb{T_\alpha}$ is normal in~$G$.
Since $\wb{T_\alpha}$ is a torsion group
(see \ref{recall3}\,(c) and \ref{torsionclosure})
and $D_\alpha$ torsion-free (see \ref{recall3}\,(c)),
we see that $D_\alpha\cap \wb{T_\alpha}=\{1\}$.
Moreover, using that $T_\alpha U_0=\wb{T_\alpha}$ by
Lemma~\ref{torsionfacts}\,(d),
we obtain $G=U_\alpha U_0=D_\alpha T_\alpha U_0=D_\alpha\wb{T_\alpha}$.
Hence $G=D_\alpha\times \wb{T_\alpha}$
as an abstract group. In particular, $D_\alpha$ centralizes
$\wb{T_\alpha}$. Thus $D_\alpha$ also centralizes
$U_0\sub \wb{T_\alpha}$.
Since $\wb{T_\alpha}$ is $\sigma$-compact by Lemma~\ref{sigmacomp},
also
$D_\alpha^* \times \wb{T_\alpha}$ is a $\sigma$-compact
locally compact group
(writing $D_\alpha^*$ for $D_\alpha$, endowed with the locally compact
topology induced by $U_\alpha^*$).
Because also~$G$ is locally compact and the product
map $\pi\colon D_\alpha^* \times \wb{T_\alpha}\to G$, $(x,y)\mto xy$
is a continuous isomorphism of abstract groups,
we deduce with \cite[5.29]{HaR}
that~$\pi$ is an isomorphism
of topological groups. Hence $G=D_\alpha\times\wb{T_\alpha}$
(internally). In particular, $D_\alpha$ is closed in~$G$.\Punkt
\begin{cor}\label{trivsect}
Let $G$ be a totally disconnected,
locally compact group and $\alpha$ be an automorphism of~$G$
such that $U_\alpha$ and $U_{\alpha^{-1}}$ can be made locally compact
$($e.g., any expansive automorphism$)$.
Then $D_\alpha\cap D_{\alpha^{-1}}=\{1\}$.
\end{cor}
\begin{proof}
Since $D_\alpha$ and $D_{\alpha^{-1}}$
are closed and $\alpha$-stable, their intersection
$H:=D_\alpha\cap D_{\alpha^{-1}}$
is a totally disconnected, locally compact contraction group
for both $\alpha|_H$ and $\alpha^{-1}|_H$.
Hence $H=\{1\}$. Indeed, if $H\not=\{1\}$
then both $\Delta_H(\alpha|_H)$
and $\Delta_H(\alpha|_H^{-1})=\Delta_H(\alpha|_H)^{-1}$
would be integers $\geq 2$ (see \ref{recall3}\,(a)),
which is impossible.
\end{proof}
\begin{rem}\label{notnormal}
We mention that the nub $U_0$ of an expansive automorphism
$\alpha\colon G\to G$ need not have an open normalizer in~$G$. To
see this, let~$F$ be a finite group which is a semidirect product
$F=N\rtimes H$ of a normal subgroup~$N$ and a subgroup~$H$ which is
not normal in~$F$ (e.g., $F$ might be the dihedral group $C_3\rtimes
C_2$). Let~$G$ be the group of all $(n_k,h_k)_{k\in\Z}\in F^\Z$ such
that $(n_k)_{k\in\Z}\in N^{(-\N)}\times N^{\N_0}=:M$. Thus
$G=M\rtimes H^\Z$ as an abstract group. Endow~$G$ with the topology
making it the direct product topological space of the restricted
product~$M$ and the compact group~$H^\Z$. Then~$G$ is a topological
group, being the ascending union of the open subgroups
\[
H^{\{k\in\Z\colon k<{-m}\}}
\times F^{\{k\in\Z\colon k\geq {-m}\}}
\]
for $m\in \N$, which are topological groups.
The right-shift~$\alpha$ is an automorphism of~$G$.
We have
$U_\alpha=M\rtimes (H^{(-\N)}\times H^{\N_0})$
and $U_{\alpha^{-1}}= H^{-\N}\times H^{(\N_0)}$.
Thus $\wb{U_\alpha}=G$,
$\wb{U_{\alpha^{-1}}}=H^\Z$
and $U_0=\wb{U_\alpha}\cap \wb{U_{\alpha^{-1}}}=H^\Z$
(using \ref{recall1}\,(b)).
Since~$H$ is not normal in~$F$,
we see that $U_0=H^\Z$ is not normal in~$G$.
If the normalizer $N_G(U_0)$ was open in~$G$,
then (being $\alpha$-stable), it would contain
the dense subgroup~$U_\alpha$
of~$G$ and hence coincide with~$G$
(a contradiction). Thus $N_G(U_0)$ is not open.
\end{rem}
\section{Abelian expansion groups}
We show that,
after passing to a refinement if necessary,
only abelian, non-discrete
groups of a special form
will occur in Theorem~B.
\begin{rem}\label{nicerabel}
In the situation of
Theorem~B,
let~$I$ be the set of all indices $j\in \{1,\ldots, n\}$
such that $G_{j-1}/G_j$ is abelian and non-discrete.
Let $\alpha_j$ be the automorphism of $G_{j-1}/G_j$
induced by~$\alpha$
and $q_j\colon G_{j-1}\to G_{j-1}/G_j$ be the
quotient homomorphism, for $j\in I$.
Then $U_{\alpha_j}U_{\alpha_j^{-1}}$
is an open $\alpha_j$-stable subgroup of $G_{j-1}/G_j$
and hence $H_j:=q^{-1}_j(U_{\alpha_j}U_{\alpha_j^{-1}})$
is an $\alpha$-stable open normal subgroup of $G_{j-1}$.
Then $G_{j-1}/H_j$ is discrete and
all stable, closed, proper subgroups of~$H_j/G_j$
are discrete.
After inserting the~$H_j$ into the series for all $j\in I$,
we may thus assume without loss of generality
that all abelian, non-discrete subfactors
$G_{j-1}/G_j$ have the property that all
of their $\alpha_j$-stable, closed, proper subgroups
are discrete, and that $G_{j-1}/G_j
=U_{\alpha_j}U_{\alpha_j^{-1}}$.
\end{rem}
Let $G_j$ be a topological group and $\alpha_j\in \Aut(G_j)$
for $j\in \{1,2\}$. We say that $(G_1,\alpha_1)$
and $(G_2,\alpha_2)$ are isomorphic if there exists
an isomorphism $\phi\colon G_1\to G_2$
of topological groups such that $\alpha_2\circ \phi=\phi\circ\alpha_1$.
\begin{prop}\label{abeliancase}
Let $A\not=\{1\}$ be an abelian, totally disconnected, locally compact
group and $\alpha\colon A\to A$ be an expansive
automorphism. Assume that $A=U_\alpha U_{\alpha^{-1}}$ and assume
that every $\alpha$-stable proper closed subgroup of~$A$ is
discrete. Then there exists a prime number~$p$ such that
$(A,\alpha)$ isomorphic to one of the following:
\begin{itemize}
\item[\rm(a)]
$\Q_p^n$ for some $n\in \N$, together
with a contractive linear
automorphism $\beta\colon \Q_p^n\to\Q_p^n$
not admitting non-trivial proper
$\beta$-stable vector subspaces;
\item[\rm(b)]
$\Q_p^n$ for some $n\in \N$, together
with $\beta^{-1}$ for a contractive linear
automorphism $\beta\colon \Q_p^n\to\Q_p^n$
not admitting non-trivial proper
$\beta$-stable vector subspaces;
\item[\rm(c)]
$C_p^{(-\N)}\times C_p^{\N_0}$
with the right-shift;
\item[\rm(d)]
$C_p^{(-\N)}\times C_p^{\N_0}$
with the left-shift;
\item[\rm(e)]
$C_p^\Z$ with the right-shift.
\end{itemize}
\end{prop}
\begin{proof}
Let $D_\alpha$ be the divisible part and $T_\alpha$
be the torsion part of~$U_\alpha$,
and define $D_{\alpha^{-1}}$ and $T_{\alpha^{-1}}$
analogously.
If $D_\alpha\not=\{1\}$, then~$D_\alpha=D_\alpha^*$
is an $\alpha$-stable closed subgroup (see Theorem~C)
which is non-discrete (see \ref{recall2}\,(a))
and thus $A=D_\alpha$.
By \ref{recall2}\,(a) and the hypotheses, $D_\alpha$
is a divisible simple contraction group
and hence of the form described in~(a)
(see \cite[Theorem~A]{GaW}).
Likewise, $A$ is of the form described in~(b)
whenever $D_{\alpha^{-1}}\not=\{1\}$.

Throughout the rest of the proof,
assume that $D_\alpha=D_{\alpha^{-1}}=\{1\}$.
Then $A=U_\alpha U_{\alpha^{-1}}=T_\alpha T_{\alpha^{-1}}$.

Since the nub $U_0$ of~$\alpha$ is an $\alpha$-stable
closed subgroup of~$A$, it either is all
of~$A$ or discrete. Being also compact,
it is finite in the latter case,
and thus $\{1\}$ is an open $\alpha$-stable
(normal) subgroup of~$U_0$.
Now \cite[Corollary 4.4]{Wi3} (proper such do not exist)
shows that $U_0=\{1\}$.

Case $U_0=\{1\}$:
Then $T_\alpha=U_\alpha=U_\alpha U_0$
and $T_{\alpha^{-1}}=U_{\alpha^{-1}}=U_{\alpha^{-1}} U_0$
are closed $\alpha$-stable subgroups of~$A$
(using \ref{recall1}\,(a)).
If $T_\alpha\not=\{1\}$,
then $T_\alpha$ is non-discrete.
Hence $T_\alpha=A$ by the hypotheses,
and this is a simple contraction group which is a torsion
group and hence of the form described in~(c)
(see \cite[Theorem~A]{GaW}). Likewise, $A$ is of the
form described in (d)
if $T_{\alpha^{-1}}\not=\{1\}$.

Case $A=U_0$: Then~$A$ is compact and is irreducible in the sense of
\cite[Definition 6.1]{Wi3} as all its proper $\alpha$-stable closed
(normal) subgroups are finite, and moreover~$A$ is infinite (as
$U_\alpha$ or $U_{\alpha^{-1}}$ is non-trivial and hence
non-discrete, being a contraction group). Hence, by
\cite[Proposition 6.3]{Wi3}, $(A,\alpha)$ is isomorphic to the
right-shift of $F^{\Z}$ for a finite simple group~$F$. Since $A$ is
abelian, $F\isom C_p$ for some~$p$ and thus $A$ is of the form
described in~(e).
\end{proof}
\begin{rem}\label{remprodker}
Let $G$ be a totally disconnected, locally compact group
and $\alpha\colon G\to G$ be an expansive automorphism.
If~$G$ is abelian, then the map
\[
\pi\colon U_\alpha^* \times U_{\alpha^{-1}}^* \to G, \quad
(x,y)\mto xy
\]
is a continuous, open homomorphism with discrete
kernel.
For non-abelian~$G$, the map still has open
image (see Lemma~\ref{basicexpa}\,(d)),
is a local homeomorphism,
and equivariant with respect to the natural
left and right actions of $U_\alpha$
and $U_{\alpha^{-1}}$, respectively.\\[3mm]
[To see this, let $V\sub G$ be a compact open subgroup
such that $V_+\cap V_-=\{1\}$
and $V=V_+V_-$ (see Lemma~\ref{basicexpa}\,(c)).
Then~$V_-$ and $V_+$ are open subgroups
of~$U_\alpha^*$ and $U_{\alpha^{-1}}^*$, respectively
(see proof of Lemma~\ref{expalcp}).
Then $\pi(V_+\times V_-)=V$
is open in~$G$ and $\pi|_{V_+\times V_-}$ is injective,
as $vw=v'w'$ for $v,v'\in V_+$, $w,w'\in V_-$
implies $v^{-1}v'=w(w')^{-1}\in V_+\cap V_-=\{1\}$
and thus $v=v'$ and $w=w'$. Since $V_+\times V_-$
is compact, $\pi$ restricts to
a homeomorphism $V_+\times V_-\to V$.
Since $\pi(gv,wh)=g\pi(v,w)h$ for $g\in U_\alpha$,
$h\in U_{\alpha^{-1}}$ and $(v,w)\in V_+\times V_-$,
also $\pi|_{g V_+\times V_-h}$ is a homeomorphism
onto an open set.]
\end{rem}
\begin{rem}
It can happen that $U_\alpha$ is closed
for an expansive automorphism~$\alpha$
of a totally disconnected, locally compact group~$G$,
but $U_{\bar{\alpha}}$ is not closed for the induced
automorphism~$\bar{\alpha}$ on $G/N$ for some
$\alpha$-stable closed normal subgroup $N\sub G$.
The following example also illustrates Remark~\ref{remprodker}.\\[3mm]
Given a non-trivial finite abelian group $(F,+)$,
consider the restricted products
$H_1:=F^{(-\N)}\times F^{\N_0}$
and $H_2:=F^{-\N}\times F^{(\N_0)}$,
with $V_1:=F^{\N_0}$ and $V_2:=F^{-\N}$,
respectively, as compact open subgroups.
Let~$\alpha$ be the right-shift on $G:=H_1\times H_2$
(i.e., on both~$H_1$ and~$H_2$).
Then~$\alpha$ is an automorphism and it is
expansive as $\bigcap_{n\in\Z}\alpha^n(V_1\times V_2)=\{0\}$.
Moreover, $U_\alpha= H_1$ and $U_{\alpha^{-1}}=H_2$ are
closed.
Also, let $\bar{\alpha}$ be the right-shift on~$F^{\Z}$.
Then
\[
q\colon G\to F^{\Z},\quad (f,g)\mto f+g
\]
is a continuous surjective homomorphism.
Restricted to the compact open subgroup $V_1\times V_2$,
the map~$q$ is an isomorphism of topological groups.
Hence~$q$ is open, has discrete kernel,
and is a quotient morphism.
Finally, $U_{\bar{\alpha}}=F^{(-\N)}\times F^{\N_0}$
is a dense proper subgroup in $F^\Z$.
Hence $U_{\bar{\alpha}}$ is not closed in~$F^\Z\cong G/\ker(q)$.
\end{rem}
Another property can be observed.
\begin{prop}
Let $G$ be a totally disconnected,
locally compact group that is abelian, and $\alpha\colon
G\to G$ be an expansive automorphism.
Then the torsion subgroup $\tor(G)$ is closed in~$G$.
\end{prop}
\begin{proof}
Since $V:=U_\alpha U_{\alpha^{-1}}$ is an open subgroup of~$G$, we
need only show that $V\cap \tor(G)=\tor(V)$ is closed. After
replacing $G$ with its $\alpha$-stable subgroup~$V$, we may
therefore assume that $G=U_\alpha U_{\alpha^{-1}}$. Since $D_\alpha$
and $D_{\alpha^{-1}}$ are torsion-free (see \ref{recall3}\,(c)) and
$D_\alpha\cap D_{\alpha^{-1}}=\{1\}$ by Corollary~\ref{trivsect}, we
deduce that $D_\alpha D_{\alpha^{-1}}$ is isomorphic to
$D_\alpha\times D_{\alpha^{-1}}$ as an abstract group and hence
torsion-free. Hence $D_\alpha D_{\alpha^{-1}}\cap T_\alpha
T_{\alpha^{-1}}=\{1\}$. Combining this with $G=U_\alpha
U_{\alpha^{-1}} =D_\alpha D_{\alpha^{-1}} T_\alpha T_{\alpha^{-1}}$,
we see that
\begin{equation}\label{givestors}
G=(D_\alpha D_{\alpha^{-1}})\times (T_\alpha T_{\alpha^{-1}})
=D_\alpha\times D_{\alpha^{-1}}\times T_\alpha T_{\alpha^{-1}}
\end{equation}
internally as an abstract group.
By~(\ref{givestors}),
the torsion subgroup of~$G$ is
$\tor(G)=T_\alpha T_{\alpha^{-1}}$.
Hence $\tor(G)$ has finite exponent (like $T_\alpha$ and $T_{\alpha^{-1}}$).
Thus also $\wb{\tor(G)}$ is a torsion group (by \ref{torsionclosure})
and thus $\tor(G)=\wb{\tor(G)}$.
\end{proof}
\section{Example: {\boldmath$p$}-adic Lie groups}
Let $\K$ be a local field and $|.|$ be an absolute value
on~$\K$ defining its topology (see \cite{Wei}).
We pick an algebraic closure $\wb{\K}$ containing~$\K$
and use the same symbol, $|.|$, for the unique extension of the absolute
value on~$\K$ to an absolute value on~$\wb{\K}$
(see \cite[Theorem~16.1]{Shi}).
If~$E$ is a finite-dimensional $\K$-vector space
and~$\beta\colon E\to E$ a $\K$-linear
automorphism, we write~$R(\beta)$
for the set of all absolute values $|\lambda|$
of zeros~$\lambda$ of the characteristic polynomial
of~$\beta$ in~$\wb{\K}$.
We let $\wt{E}_\lambda\sub E\otimes_\K \wb{\K}$ be the generalized eigenspace
of $\beta\otimes_\K \id_{\wb{\K}}$ for the eigenvalue~$\lambda$.
For $\rho\in R(\beta)$, we let
\[
E_\rho:=
\left(\bigoplus_{|\lambda|=\rho} \wt{E}_\lambda\right)\cap E.
\]
Then $E=\bigoplus_{\rho\in R(\beta)}E_\rho$ (see
\cite[Chapter~II, \S1]{Mar})
and we recall that
\[
E=U_\beta\oplus M_\beta\oplus U_{\beta^{-1}}
\]
with
\begin{equation}\label{formsubsp}
M_\beta=E_1, \quad
U_\beta=\bigoplus_{\rho<1}E_\rho\quad
\mbox{and}\quad  U_{\beta^{-1}}=\bigoplus_{\rho>1}E_\rho
\end{equation}
(cf.\ \cite[Lemma~2.5]{GFi}).\\[2.5mm]
If $G$ is a Lie group over~$\K$,
then its tangent space $L(G):=T_1(G)$ at the identity
element
carries a natural Lie algebra structure,
and $L(\alpha)\colon L(G)\to L(H)$ is a Lie algebra
homomorphism for each $\K$-analytic homomorphism
$\alpha\colon G\to H$ between $\K$-analytic
Lie groups. We abbreviate $\Ad(g):=L(I_g)$,
where $I_g\colon G\to G$, $x\mto gxg^{-1}$ for $g\in G$
(cf.\ \cite{Ser} for further information).
\begin{prop}\label{algnilp}
Let $\alpha$ be an analytic automorphism
of a Lie group~$G$ over a local field.
Then the following conditions are equivalent:
\begin{itemize}
\item[\rm(a)]
$\alpha$ is expansive;
\item[\rm(b)]
$\beta:=L(\alpha)\colon L(G)\to L(G)$ is expansive;
\item[\rm(c)]
$1\not\in R(\beta)$.
\end{itemize}
If $\alpha$ is expansive,
then $L(G)$ is a nilpotent Lie algebra.
\end{prop}
\begin{proof}
(a)$\impl$(b): By contraposition.
If (b) is false, then~$\beta$ is not expansive.
To deduce that~$\alpha$ is not expansive,
let $V\sub G$ be an identity neighbourhood.
Since~$U_\beta$ is a vector subspace of~$L(G)$ by~(\ref{formsubsp})
and hence closed,
using Proposition~\ref{givesexp}\,(b)
we deduce that~$M_\beta$ is not discrete
and hence a non-trivial vector subspace
(in view of (\ref{formsubsp})).
But then~$G$ contains a so-called centre manifold~$W$
around the fixed point~$1$ of~$\alpha$,
which can be chosen as a
submanifold of~$G$ contained in~$V$
that is stable under~$\alpha$
and satisfies $T_1(W)=M_\alpha$ (whence $W\not=\{1\}$);
see Proposition 6.3\,(a) and part (b) of the Local Invariant Manifold Theorem in \cite{GFi};
cf.\ also \cite{GIM}.
Then $\{1\}\not=W\sub\bigcap_{n\in\Z}\alpha^n(V)$,
and thus $\alpha$ is not expansive.\\[2.5mm]
(b)$\aeq$(c): Since $U_\beta$ is closed,
$\beta$~is expansive if and only if $M_\beta=L(G)_1$ is discrete.
Since $L(G)_1$ is a vector space,
the latter holds if and only if $L(G)_1=\{0\}$,
i.e., $1\not\in R(\beta)$.\\[2.5mm]
(c)$\impl$(a):
$U_\alpha$ and $U_{\alpha^{-1}}$
are immersed Lie subgroups of~$G$ with Lie algebras~$U_\beta$
and $U_{\beta^{-1}}$, respectively
(see \cite[Theorem D]{GFi}).
We write $U_\alpha^*$
for $U_\alpha$ as a Lie group;
because the underlying topology is locally
compact and $\alpha$ restricts to a contractive
Lie group automorphism,
this is consistent with the definition
of $U_\alpha^*$ in Section~\ref{secrefine}.
Likewise, we consider $U_{\alpha^{-1}}^*$ as
a Lie group.
If $1\not\in R(\beta)$,
then $L(G)=U_\beta\oplus U_{\beta^{-1}}$,
entailing that the product map
$\pi\colon U_\alpha^*\times U_{\alpha^{-1}}^*\to G$, $(x,y)\mto xy$
has invertible differential at~$(1,1)$ (the addition map
$U_\beta\times U_{\beta^{-1}}\to L(G)=U_\beta\oplus U_{\beta^{-1}}$).
Thus, by the inverse function theorem~\cite{Ser},
there exist open identity neighbourhoods $V\sub U_\alpha^*$
and $W\sub U_{\alpha^{-1}}^*$ such that $VW$ is open in~$G$ and
the restriction
\begin{equation}\label{productmap}
\pi|_{V\times W}\colon V\times W\to VW
\end{equation}
is an analytic diffeomorphism.
Using \cite[Lemma~3.2\,(i)]{Sie},
we find compact open subgroups $P\sub V$ of $U_\alpha^*$
and $Q\sub W$ of $U_{\alpha^{-1}}^*$
such that $\alpha(P)\sub P$, $\alpha^{-1}(P)\sub V$,
$\alpha^{-1}(Q)\sub Q$ and $\alpha(Q)\sub W$.
Then $PQ$ is an open identity neighbourhood
in~$G$
and we now show that $\bigcap_{n\in\Z}\alpha^n(PQ)=\{1\}$.
To this end, let $x\in P$ and $y\in Q$.
If $x\not=1$, then $\alpha^{-n}(x)\not\in P$
for some $n\in \N$, which we choose minimal.
Thus $\alpha^{-n+1}(x)\in P$
and hence $\alpha^{-n}(x)\in V$.
Since $\alpha^{-n}(x)\in V\setminus P$ by the preceding
and $\alpha^{-n}(y)\in Q$, we see that
$\alpha^{-n}(xy)=\alpha^{-n}(x)\alpha^{-n}(y)\in (V\setminus P)Q$.
As the map (\ref{productmap}) is a bijection,
we deduce that $\alpha^{-n}(xy) \not \in PQ$.
Likewise, $\alpha^m(xy)\not \in PQ$
for some $m\in\N$ if $y\not=1$.
Thus $\bigcap_{n\in\Z}\alpha^n(PQ)=\{1\}$ indeed and thus~$\alpha$
is expansive.\\[2.5mm]
\emph{Final assertion.}
If (c) holds, then $\beta$ is a Lie algebra automorphism of~$L(G)$
and $|\lambda|\not=1$
for all eigenvalues~$\lambda$
of $\beta\otimes_\K\id_{\wb{\K}}$ in~$\wb{\K}$,
entailing that none of the $\lambda$ is a root of unity.
Hence $L(G)$ is nilpotent
(see Exercise~21\,(b) among the exercises
for Part~I of \cite{Bou}, \S4, or
\cite[Theorem 2]{Jac}).
\end{proof}
\begin{numba}\label{ratner}
For each continuous homomorphism
$\theta\colon \Q_p\to \GL_n(\Q_p)$,
there exists a nilpotent $n\!\times \!n$-matrix
$x\in \Q_p^{n\times n}$
such that $\theta(t)=\exp(tx)$
for all $t\in \Q_p$, using the
matrix exponential function
\cite[Theorem~1.1]{Rat}.
Thus $\theta'(0)=x$ uniquely determines~$\theta$,
and so does $\theta|_W$ for any $0$-neighbourhood $W\sub \Q_p$.
\end{numba}
\begin{la}\label{onepars}
Let $\alpha$ be a contractive automorphism
of a $p$-adic Lie group~$G$.
Then the following holds:
\begin{itemize}
\item[\rm(a)]
For each $g\in G$, there is a
unique continuous homomorphism $\theta_g\colon \Q_p\to G$
such that $\theta_g(1)=g$.
Moreover, $\{\theta_g'(0)\colon g\in G\}=L(G)$.
\item[\rm(b)]
If $\ch\sub L(G)$
is an $L(\alpha)$-stable Lie subalgebra,
then there exists an
$\alpha$-stable
Lie subgroup~$H$ of~$G$ with $L(H)=\ch$.
\end{itemize}
\end{la}
\begin{proof}
(a) Let $*\colon L(G)\times L(G)\to L(G)$
be the Campbell-Hausdorff multiplication on the
nilpotent Lie algebra~$L(G)$.
Because $(G,\alpha)$ and $((L(G),*),L(\alpha))$
are locally isomorphic contraction groups,
they are isomorphic (see \cite[Proposition~2.2]{Wan}).
The nilpotent group $(L(G),*)$ inherits unique
divisibility from $(L(G),+)$,
since $ng$ (in the vector space $L(G)$)
coincides with $g^n$ (in $(L(G),*)$).
It is clear from this that $\theta_g(t)=tg$
is the unique continuous homomorphism $\Q_p\to (L(G),*)$
with $\theta_g(1)=g$. It satisfies $g=\theta_g'(0)$.\\[2.5mm]
(b) We may work with the isomorphic group
$(L(G),*)$ instead of~$G$.
Now $H:=\ch$ is an $L(\alpha)$-stable
Lie subgroup of $(L(G),*)$ with Lie algebra~$\ch$.
\end{proof}
\begin{la}\label{nwla}
Let $G$ be a linear $p$-adic Lie group.
Assume that $G$ is generated by $\bigcup_{\theta\in \Theta}\theta(\Q_p)$
for a set~$\Theta$ of continuous homomorphisms $\theta\colon \Q_p\to G$,
and $L(G)$ is generated by $\{\theta'(0)\colon \theta\in \Theta\}$
as a Lie algebra. Then the centre of~$G$ coincides with the kernel
of $\Ad\colon G\to \Aut(L(G))$.
\end{la}
\begin{proof}
Let $g\in G$.
For each $\theta\in \Theta$,
the map $I_g\circ \theta \colon \Q_p\to G$, $t\mto g\theta(t)g^{-1}$
is a continuous homomorphism such that
$(I_g\circ \theta)'(0)=\Ad(g)\theta'(0)$.
Thus, by \ref{ratner},
$I_g\circ \theta=\theta$ if and only if $\Ad(g)\theta'(0)=\theta'(0)$.
Since $\bigcup_{\theta\in \Theta}\theta(\Q_p)$ generates~$G$,
we see that $g\in Z(G)$ if and only if $\Ad(g)\theta'(0)=\theta'(0)$
for all $\theta\in\Theta$. The latter is equivalent
to $\Ad(g)(x)=x$ for all $x\in L(G)$,
because $\{x\in L(G)\colon \Ad(g)(x)=x\}$
is a Lie subalgebra of~$L(G)$
and $L(G)$ is generated by~$\theta'(0)$ for $\theta\in \Theta$
by hypothesis.
\end{proof}
{\bf Proof of Theorem~D.}
After replacing $G$ with an open subgroup, we may assume
that~$G$ is generated by $U_\alpha\cup U_{\alpha^{-1}}$ (see
Lemma~\ref{basicexpa}\,(d)).
We prove that~$G$ is nilpotent in this case,
by induction on the dimension
$\dim(G)$ of~$G$ as a $p$-adic manifold.
If $\dim(G)=0$, then~$G$ is discrete, whence $U_\alpha=U_{\alpha^{-1}}=\{1\}$
and $G=\langle U_\alpha\cup U_{\alpha^{-1}}\rangle=\{1\}$
is nilpotent.\\[2.5mm]
Now assume that $\dim(G)>0$. After replacing~$G$ with an isomorphic
group, we may assume that $G$ is a subgroup of~$\GL_n(\Q_p)$ for
some $n\in\N$, and that the inclusion map $G\to \GL_n(\Q_p)$ is
continuous (but not necessarily a homeo\-morphism onto its image).
Then $L(G)$ is a non-zero nilpotent Lie
algebra (see (a)$\impl$(d) in Proposition~\ref{algnilp})
and so it has
centre $Z(L(G))\not=\{0\}$. The centre is $L(\alpha)$-stable, and
the restriction~$\beta$ of $L(\alpha)$ to the centre is expansive
(like $L(\alpha)$). Hence $Z(L(G))=U_\beta\oplus U_{\beta^{-1}}$.
After replacing~$\alpha$ with~$\alpha^{-1}$ if necessary, we may
assume that $U_{\beta^{-1}}\not=\{0\}$. According to
Lemma~\ref{onepars}\,(b), there is an $\alpha$-stable Lie subgroup
$H\sub U_{\alpha^{-1}}$ with $L(H)=U_\beta$. We claim that~$H$ is in
the centre~$Z(G)$ of~$G$. If this is true, then
$Z(G)$ has positive dimension.
Thus $G/Z(G)$ is a Lie group of dimension $\dim(G/Z(G))<\dim(G)$,
and it is a linear Lie group as it injects into $\Aut(L(G))$,
by Lemma~\ref{nwla}.\footnote{Let $\Theta$ be the set
of continuous homomorphisms from $\Q_p$ to $U_\alpha$
or $U_{\alpha^{-1}}$.
By $G=\langle U_\alpha\cup U_{\alpha^{-1}}\rangle$
and Lemma~\ref{onepars}\,(a),
the first hypothesis of Lemma~\ref{nwla} is satisfied.
Since $L(G)=L(U_\alpha)+L(U_{\alpha^{-1}})$
and $L(U_\alpha)\cup L(U_{\alpha^{-1}})
=\{\theta'(0)\colon \theta\in \Theta\}$
by Lemma~\ref{onepars}\,(a) and ``(a)$\impl$(c)'' in Lemma~\ref{algnilp},
also the second hypothesis of Lemma~\ref{nwla} is satisfied.}
By induction,
$G/Z(G)$ is nilpotent and hence so is~$G$.\\[2.5mm]
To prove the claim, let $h\in H$
and $\theta \colon \Q_p\to H\sub
U_{\alpha^{-1}}$ be a
continuous homomorphism with $\theta(1)=h$
(see Lemma~\ref{onepars}\,(a)).
Then $x:=\theta'(0)\in L(H)\sub Z(L(G))$,
entailing that $\ad(x):=[x,\sbull]=0$.
Now $\theta(t)=\exp(tx)$ for all
$t\in \Q_p$, by~\ref{ratner}.
For $|t|$ small, $\Ad(\theta(t))=\Ad(\exp(tx))=e^{t\ad(x)}=\id_{L(G)}$,
using Corollary~3
in \cite[Chapter III, \S4, no.\,4]{Bou}).
Thus $\Ad\circ \theta=\id_{L(G)}$,
by the uniqueness assertion of \ref{ratner},
applied to $\Ad\circ\theta\colon \Q_p\to \Aut(L(G))$.
In particular, $\Ad(h)=\Ad(\theta(1))=\id_{L(G)}$
and thus $h\in Z(G)$, by Lemma~\ref{nwla}.\Punkt
\begin{numba}\label{nusitu}
If $G$ is a totally disconnected, locally compact group which is a
nilpotent group, let $\{1\}=Z_0\lhd Z_1\lhd\cdots \lhd Z_n=G$ be its
ascending central series defined recursively via
$Z_k:=q_k^{-1}(Z(G/Z_{k-1}))$, where $q_k\colon G\to G/Z_{k-1}$ is
the canonical quotient morphism. Let~$\alpha$ be an expansive
automorphism of~$G$ and~$\alpha_k$ the induced automorphism of
$G_k/G_{k-1}$.
\end{numba}
\begin{prop}\label{issubgp}
If $Z_k/Z_{k-1}=U_{\alpha_k}U_{\alpha_k^{-1}}$
for all $k\in \{1,\ldots, n\}$ in the situation of \emph{\ref{nusitu}},
then $G=U_\alpha U_{\alpha^{-1}}$.
In particular, $U_\alpha U_{\alpha^{-1}}$
is a subgroup of~$G$.
\end{prop}
\begin{proof}
If $n=0$, then $G=\{1\}=U_\alpha U_{\alpha^{-1}}$.
If $n\geq 1$, let $\beta$ be the expansive
automorphism of $G/Z(G)$ induced by~$\alpha$,
and $q\colon G\to G/Z(G)$ be the canonical quotient morphism.
Then $Z_1=Z(G)=(U_\alpha\cap Z(G)) (U_{\alpha^{-1}}\cap Z(G))$
by the hypotheses and $G/Z(G)=U_\beta U_{\beta^{-1}}$
by induction.
Since $q(U_\alpha U_{\alpha^{-1}})=U_\beta U_{\beta^{-1}}=G/Z(G)$,
we have $G=U_\alpha U_{\alpha^{-1}} Z(G)=
U_\alpha U_{\alpha^{-1}}
(U_\alpha\cap Z(G)) (U_{\alpha^{-1}}\cap Z(G))
=
U_\alpha
(U_\alpha\cap Z(G))
U_{\alpha^{-1}}
(U_{\alpha^{-1}}\cap Z(G))
=U_\alpha U_{\alpha^{-1}}$.
\end{proof}
Note that we can easily achieve that
$G/Z_{n-1}=U_{\alpha_n}U_{\alpha_n^{-1}}$ after replacing~$G$ with
its open subgroup generated by $U_\alpha\cup U_{\alpha^{-1}}$.
However, the hypotheses on $Z_k/Z_{k-1}$ for $k<n$ cannot always be
achieved by passing to an open subgroup (as the following example
illustrates).
\begin{rem}\label{Hei1} The following example shows that even for nilpotent
$p$-adic Lie groups with an expansive automorphism~$\alpha$,
the set $U_\alpha U_ {\alpha ^{-1}}$ may fail to be a subgroup.
The example also
provides a $p$-adic Lie group that admits
expansive automorphisms but does not
admit any contractive automorphism. In fact,
the group has a closed discrete commutator group
which is characteristic and hence would inherit a contractive
automorphism (contradicting the fact that non-trivial contraction groups
are non-discrete).\\[2.4mm]
Let $H = \Q_p ^3$ be the $3$-dimensional $p$-adic Heisenberg group
whose binary operation is given by
\[
(x _1, y_1,z _1) (x_2, y_2, z_2) = (x_1+x_2 , y_1 + y _2 ,
z_1+z_2+x_1 y_2)
\]
for all $(x _1, y_1,z _1), (x_2, y_2, z_2) \in H$. Let $N= \{ (0,0,
z ) \in H \colon |z| \leq 1 \}$.  Then~$N$ is a compact central
subgroup of $H$. Identify $G= H/N$ with $\Q_p\times\Q_p\times
(\Q_p/\Z_p)$ as a set.  Define $\alpha \colon G \to G$ by
\[
\alpha (x,y, z+\Z_p) = (px , p^{-1}y, z+\Z_p)
\]
for all $(x,y, z+\Z_p)\in G$.
Then $\alpha$ is a continuous automorphism of the $p$-adic Lie group
$G$ with
$M_\alpha = \{ (0,0, z+\Z_p) \colon  z\in \Q _p \}$, $U_\alpha =
\{ (x ,0 , 0) \colon  x\in \Q _p \}$
and $U_{\alpha ^{-1}} = \{(0,y ,
0) \colon y \in \Q _p \}$.
Since $M_\alpha$ is discrete, $\alpha$ is
an expansive automorphism.  As $[U_\alpha , U_{\alpha ^{-1}}] = \{
(0,0, z+\Z_p)\colon z \in \Q_p\}$ and $U_\alpha U_{\alpha ^{-1}} =
\{(x,y, xy+\Z_p) \colon x, y \in \Q _p \}$, we get that $U_\alpha
U_{\alpha ^{-1}}$ is a not a subgroup.
\end{rem}
\begin{prop}\label{alggp}
Let $G$ be a closed subgroup of $\GL_n(\Q _p)$ and~$\alpha$ be an
expansive automorphism of $G$. Then $U_\alpha U_{\alpha ^{-1}}$ is
an open $($unipotent algebraic$)$ subgroup of~$G$.
\end{prop}
\begin{proof}
Replacing $G$ by the group generated by~$U_\alpha$ and~$U_{\alpha^{-1}}$,
we may assume by Theorem~E that~$G$ is a closed nilpotent
subgroup of $\GL_n(\Q_p)$.  Let~$\mathbb G$ be the Zariski closure
of $G$.  Then~$\mathbb G$ is defined over~$\Q_p$ and~$\mathbb G$ is
nilpotent (cf.\ Proposition~1.3\,(b) and Corollary~1 in~2.4
of~\cite{Bo}).  Since~$U_\alpha$ and~$U_{\alpha^{-1}}$ consists of
one-parameter (unipotent) subgroups, $\mathbb G$ is
Zariski-connected.  This implies that the set of unipotent elements
form a subgroup $\mathbb G_u$, known as the unipotent radical (cf.\
Theorem~10.6 of \cite{Bo}).  Since $U_\alpha$ and $U_{\alpha^{-1}}$
consists of one-parameter (unipotent) subgroups, $U_\alpha ,
U_{\alpha ^{-1}} \sub \mathbb G_u$.  This implies that $\mathbb G
= \mathbb G_u$, that is~$\mathbb G$ is an unipotent algebraic group,
hence $\mathbb G$ is defined over~$\Q _p$ (cf.\ 4.5 of \cite{Bo} and
the fact that $\Q_p$-closed and defined over $\Q_p$ are same as
characteristic of $\Q_p$ is zero) and $G\sub \mathbb G (\Q_p)$.\\[2.4mm]
For $i\geq 1$, let $D_i=[\mathbb G (\Q_p), D_{i-1}]$ with
$D_0=\mathbb G(\Q _p)$ and $G_i = \overline {[G, G_{i-1}]}$ with
$G_0=G$. Then $D_{k+1}$ is trivial for some $k\geq 1$ as $\mathbb G$
is unipotent and $G_i\subset D_i$.  Thus, $G_k$ is a closed
$\alpha$-stable subgroup of~$D_k$ which is a vector space. Let~$V$
be the maximal vector subspace of~$G_k$.  Then~$V$ is a closed
$\alpha$-stable central subgroup of~$G$.  The automorphism $\beta
\colon G_k /V \to G_k/V$ defined by $\beta (x+V) = \alpha (x) +V$
for $x\in G_k$ is expansive. Since~$V$ is the maximal vector
subspace of~$G_k$ and~$G_k$ is a closed subgroup of the $p$-adic
vector space~$D_k$, we get that~$G_k/V$ is a compact subgroup of the
$p$-adic vector space~$D_k/V$. Since the automorphism group of a
compact $p$-adic analytic group is compact, compact $p$-adic
analytic groups do not admit expansive automorphisms unless finite,
hence $V=G_k$.  This implies that $G_k= D_k$ and $G_k=V=
(U_\alpha\cap V) (U_{\alpha ^{-1}}\cap V)$. Since $G/G_k$ is a
closed subgroup of $\mathbb G(\Q _p)/D_k$ which is a linear
($p$-adic algebraic) group, the result follows by induction.
\end{proof}
\begin{rem}\label{heis2}
In the case of linear $p$-adic Lie groups, even if $U_\alpha
U_{\alpha  ^{-1}}$ is an open subgroup for an expansive automorphism,
the following example shows that it is not possible to have either
of $U_{\alpha }$ or $U_{\alpha ^{-1}}$ to normalize the other.\\[2.4mm]
Let~$H$ be the $3$-dimensional $p$-adic Heisenberg group defined as
in Remark~\ref{Hei1}.
For $i=1,2$, define $\alpha _i \colon H \to H$ by
\[
\alpha _1(x,y,z) = (px , p^{-2}y, p^{-1}z), ~~\alpha _2(x,y, z) = (p^2x,
p^{-1}y, pz)
\]
for $(x,y, z)\in H$.  Let $G= H\times H$ and $\alpha
= \alpha _1 \times \alpha _2$.  Then
\[
U_{\alpha } = \{(x,0, 0) \colon
x\in \Q _p \} \times \{(a,0, c) \colon a,c\in \Q _p \}
\]
and
\[
U_{\alpha ^{-1}} = \{(0,y,z) \colon y,z\in \Q _p \} \times \{(0,b, 0)
\mid b \in \Q _p \}.
\]
Thus $U_\alpha U_{\alpha^{-1}}=G$.
Since $\{ (x,0,0) \colon x \in \Q_p \}$ and
$\{ (0,y,0) \colon y \in \Q _p \}$ are not normal subgroups of $H$,
neither $U_\alpha$ or $U_{\alpha ^{-1}}$ normalize the other.
\end{rem}
\section{Example: Baumslag-Solitar groups}
Throughout this section, we fix primes $p\not=q$.
We let
\[
\BS(p,q):=\langle a,t| ta^pt^{-1}=a^q\rangle
\]
be the Baumslag-Solitar group. Then $\langle a\rangle \cap g\langle
a\rangle g^{-1}$ has finite index in $\langle a\rangle$ for each
$g\in \BS(p,q)$, and $\bigcap_g g\langle a \rangle g^{-1} =\{ 1 \}$,
hence the Schlichting completion $G_{p,q}$ of $\BS(p,q)$ can be
formed, which is a certain totally disconnected, locally compact
group in which~$\BS(p,q)$ is dense, and in which $K:=\wb{\langle
a\rangle}$ is a compact open subgroup (see \cite{EaW}, cf.\
\cite{Sch} and \cite{Hec}). We are interested in the inner
automorphism\vspace{-3mm}
\[
\alpha\colon G_{p,q}\to G_{p,q},\quad x\mto txt^{-1}\,.\vspace{1mm}
\]
{\bf Proof of Theorem E.}
By \cite[Proposition~8.1]{EaW}, $K$ contains an open subgroup $V\cong
\Z _p \times \Z _q$ and $K/V$ is a cyclic group of order diving ${\rm
gcd}(p,q)=1$.  Thus $K=V \cong \Z_p \times \Z_q$.
After multiplication with a unit,
We may assume that the isomorphism takes $a$ to~$(1,1)$.\\[2.3mm]
Let $G= \Z \ltimes  (\Q _p \times \Q _q)$ be the semidirect product
of $\Z$ and $\Q _p \times \Q _q$ given by $(n,u,v)(m,u',v')= (n+m,
u+(q/p)^n u', v+(q/p)^nv')$ for all $n,m\in \Z$, $u,u'\in \Q_p$ and
$v,v'\in \Q _q$.  The isomorphism $K\cong \Z_p \times \Z_q$ gives
a homomorphism from $\langle a \rangle$ to~$G$. Since
$(1,0,0)(0, p, p)(-1,0,0) = (0,q,q)$, sending $t\mto (1,0,0)$ yields
a group homomorphism $\phi \colon \BS(p,q) \to G$.  Since $\phi |_{\langle a
\rangle}$ is a continuous homomorphism, $\phi$ extends to a
continuous homomorphism of $G_{p,q}$ into $G$ which would also be
denoted by $\phi$.  Since $\phi |_K$ is an isomorphism, $\ker(\phi)$
is discrete. Moreover, as $\phi (G_{p,q})$ contains
both $(1,0,0)$ and $\Z _p\times \Z _q$, $\phi$ is surjective.\\[2.3mm]
Let~$\beta$ be the inner automorphism of~$G$ given by $(1,0,0)$.
Then $\phi \circ \alpha = \beta \circ \phi$ and~$\beta$ is expansive.
Since the
kernel of~$\phi$ is discrete, expansiveness of~$\alpha$ follows from
Theorem~A.
As the open subgroup $K\cong \Z_p\times \Z_q$
satisfies an ascending chain condition on closed
subgroups (see, e.g., \cite[Proposition~3.2]{Grp}), $U_\alpha$ is closed
by \cite[Lemma~3.2]{Wan}.\\[2.3mm]
In case $U_\alpha U_{\alpha^{-1}}$ is a group, we will now show
that~$\phi$ is an isomorphism which would lead to a contradiction as
$G_{p,q}$ is not solvable but~$G$ is solvable.\footnote{Alternatively,
$\BS(p,q)$ would be a finitely generated
linear group then and hence residually finite by~\cite{Mal}.
But $\BS(p,q)$ is not residually finite for primes
$p\not=q$, see \cite{Mes}.}
Suppose $N:=U_\alpha
U_{\alpha ^{-1}}$ is a group.
Then~$\phi |_N$ is an isomorphism of~$N$ with $\Q_p \times \Q_q$
(using that $U_\alpha\cong\Q_q$ and
$U_{\alpha^{-1}}\cong \Q_p$).\footnote{In fact, $V=K$ is tidy for~$\alpha$ with $V_-\cong \Z_q$,
$V_+\cong \Z_p$ and $V_0=\{1\}$ (see \cite{EaW}),
whence $U_\alpha=V_{--}\cong\Q_q$
and $U_{\alpha^{-1}}=V_{++}\cong\Q_p$ (cf.\ Lemma~\ref{basicexpa}\,(b).}
Now the group generated by $t$ and~$N$ is an open subgroup of~$G$
containing both~$t$ and~$a$, hence $G_{p,q}=\langle t, N \rangle
=\langle t\rangle N$ (as $t$ normalizes~$N$).
This implies that~$\phi$ is an isomorphism.\,\Punkt
{\small
{\bf Helge  Gl\"{o}ckner}, Universit\"at Paderborn, Institut f\"{u}r Mathematik,\\
Warburger Str.\ 100, 33098 Paderborn, Germany;\,
e-mail: {\tt  glockner\at{}math.upb.de}\\[2.5mm]
{\bf C.\,R.\,E. Raja},
Stat-Math Unit, Indian Statistical Institute (ISI),\\
8th Mile Mysore Road, Bangalore 560 059, India;\,
e-mail: {\tt creraja@isibang.ac.in}}
\end{document}